\documentclass[12pt,twoside]{amsart}
\usepackage{graphicx,mathrsfs,epstopdf}
\usepackage{amsmath,amsfonts,amssymb,amsthm}
\usepackage[colorlinks,citecolor=blue,urlcolor=blue]{hyperref}
\usepackage{tikz}
\usepackage[font=small,labelfont=bf]{caption}
\usetikzlibrary{positioning}
\usepackage{bbold}
\usepackage{mathtools}
\mathtoolsset{showonlyrefs,showmanualtags}

\newtheorem{thm}{Theorem}[section]
\newtheorem{lem}[thm]{Lemma}
\newtheorem{prop}[thm]{Proposition}
\newtheorem{cor}[thm]{Corollary}

\theoremstyle{definition}
\newtheorem{defn}[thm]{Definition}

\newtheorem{conj}[thm]{Conjecture}
\newtheorem{exmp}[thm]{Example}

\theoremstyle{remark}
\newtheorem{rem}[thm]{Remark}



\DeclareMathOperator{\tr}{tr}

\newcommand{\C}{\mathbb{C}}

\newcommand{\R}{\mathbb{R}}
\newcommand{\RR}{\mathbb{R}}

\title{Positivity Certificates via Integral Representations}
\author{
Khazhgali Kozhasov, Mateusz Micha\l ek and Bernd Sturmfels}
\email{}
\thanks{}

\begin{document}

\begin{abstract}
Complete monotonicity is a strong positivity property for
real-valued functions on convex cones. It is certified by the
kernel of the inverse Laplace transform.
We study this for negative powers of hyperbolic polynomials.
Here the certificate is the  Riesz kernel
in G{\aa}rding's integral representation.
The Riesz kernel is a hypergeometric function in the
coefficients of the given polynomial. For monomials
in linear forms, it is a  Gel'fand-Aomoto hypergeometric function,
related to volumes of polytopes.
We establish complete monotonicity for sufficiently negative powers of
elementary symmetric functions. We also show that small negative powers of these polynomials are not completely monotone, proving one direction of a conjecture
by Scott and~Sokal. 
\end{abstract}

\maketitle

\section{Introduction}
\label{sec1}
A real-valued function   $\,f : \R^n_{>0} \rightarrow \R \,$ 
on the positive orthant is {\em completely monotone}~if
\begin{align}\label{eq:cm0}
  (-1)^k\frac{\partial^kf}{\partial x_{i_1} \partial x_{i_2}  \cdots \,\partial x_{i_k}}(x) \,\, \geq 
  \,\, 0 \quad \hbox{   for all $x\in \R_{>0}^n$,}
\end{align}
and for all index sequences $1 \leq i_1 \leq i_2 \leq \cdots \leq i_k \leq n$
 of arbitrary length $k$.
 In words, the function $f$ and all of its
 signed derivatives are nonnegative on the open cone~$\R_{>0}^n$.
 
While this definition makes sense for $\,\mathcal{C}^{\infty}$ functions $f$,
we here restrict ourselves to a setting that is the natural one in
 real algebraic geometry. Namely, we consider functions 
\begin{equation} \label{eq:factored}
 f \,\,= \,\, p_1^{s_1} p_2^{s_2} \cdots p_m^{s_m}  , 
 \end{equation}
where each $p_i$ is a homogeneous polynomial in $n$ variables that is positive 
on $\R^n_{>0}$.
 The alternating sign in \eqref{eq:cm0}
implies that positive powers of a polynomial cannot be completely monotone.
The real numbers $s_1,s_2,\ldots,s_m$ will therefore be negative in most cases.

The signed $k$-th order derivative in \eqref{eq:cm0} for $f=p^{s}$ has the form 
$P_{i_1\cdots i_k} \cdot p^{s-k}$ where $P_{i_1\cdots i_k}$
is a homogeneous polynomial.
Saying that $f$ is completely monotone means that
infinitely many polynomials $P_{i_1 \cdots i_k}$ 
are nonnegative on $\RR_{>0}^n$. How can this be certified?

To answer this question we apply the inverse Laplace transform
to the function $f(x)$. Our goal is to find a 
function $q(y)$ that is nonnegative on $\RR_{>0}^n$ and that satisfies
\begin{equation}
\label{eq:inverseLaplace}
 \qquad  f(x) \,\,= \, \int_{\RR_{\geq 0}^n}\!\! e^{- \langle y,x \rangle }\, q(y) dy
\,\,\, \quad \hbox{   for all $x\in \R_{>0}^n$}. 
\end{equation}
Here $dy$ denotes Lebesgue measure and
 $\langle y,x \rangle = \sum_{i=1}^n y_i x_i$ is the usual dot product.
This  integral representation certifies that $f(x)$ is completely monotone, since it implies
$$ (-1)^k\frac{\partial^kf}{\partial x_{i_1} \partial x_{i_2} \cdots \,\partial x_{i_k}}(x) \,
\,\, =\,\,\,  \int_{\RR_{\geq 0}^n} \! e^{- \langle y,x \rangle } y_{i_1} y_{i_2} \cdots y_{i_k} \, q(y) dy
\,\,\, \geq \,\, \, 0.
$$

Our interest in the formula \eqref{eq:inverseLaplace}
derives from statistics and polynomial optimization.
Certifying that a polynomial is nonnegative is an essential primitive in optimization. 
This can be accomplished by finding a representation as a
sum of squares (SOS) or as a sum of nonnegative circuit polynomials (SONC).
See e.g.~\cite{BPT} and \cite{DIW} for introductions to these techniques.
We here introduce an approach that may lead to new  certificates.

Algebraic statistics \cite{Seth} is concerned with probabilistic models
that admit polynomial representations and are well suited for data analysis. The
{\em exponential families}  in \cite{MSUZ2016} 
enjoy these desiderata.  Using the formulation above,
the distributions in an exponential family have
supports in $\RR_{>0}^n$ and
probability density functions $\,y \mapsto e^{- \langle y,x \rangle } \frac{q(y)}{f(x)}$.

One natural class of instances $f$ arises  when
the factors $p_1,\ldots,p_m$ in \eqref{eq:factored}
are linear forms with nonnegative coefficients.
An easy argument, presented in Section~\ref{sec3},
furnishes the  integral representation \eqref{eq:inverseLaplace}
and shows that $f$ is completely monotone.
However, even this simple case is important for applications, notably in the 
algorithms  for  efficient volume computations due to Lasserre and Zerron \cite{LZ1, LZ2}.
While that work was focused on polyhedra, the theory
presented here promises a natural generalization
to spectrahedra, spectrahedral shadows, and feasible regions in hyperbolic programming.

\smallskip

The present article is organized as follows. In Section \ref{sec2}
we replace $\RR_{> 0}^n$ by an arbitrary convex cone $C$, and we 
recall the Bernstein-Hausdorff-Widder-Choquet  Theorem.
This theorem says that a certifcate \eqref{eq:inverseLaplace} exists for
every completely monotone function on $C$, provided we replace
$ q(y) dy $ by a Borel measure supported in the dual cone $C^*$.
We also discuss operations that preserve complete monotonicity, such as 
restricting $f(x)$ to a linear subspace. Dually, this corresponds to pushing forward 
the representing measure.

In Section \ref{sec3} we study the case of monomials in linear forms,
seen two paragraphs above. Here the function $q(y)$ is  piecewise
polynomial, for integers $s_1,\ldots,s_m < 0$, it measures
the volumes of  fibers under projecting of a simplex onto a polytope.
In general, we interpret $q(y)dy$ as pushforward of a Dirichlet distribution
with concentration parameters $-s_i$, thus offering a link to
Bayesian statistics and machine learning.

Section \ref{sec4} concerns the  $m = 1$ case in \eqref{eq:factored}.
In order for $f = p^s$ to be completely monotone, the polynomial $p$
must be hyperbolic and the exponent $s$ must be negative.
The function $q(y)$ is known as the {\em Riesz kernel}. It can be
computed from the complex integral representation in
Theorem \ref{thm:Riesz} which is due to G{\aa}rding \cite{garding1951}.
Following \cite[Section 3]{MSUZ2016} and \cite[Section 4]{scott2014},  
we study the case when $p$ 
is the determinant of a symmetric matrix of linear forms.
Here the Riesz kernel is related to the Wishart~distribution.

Our primary objective is to develop tools for computing
Riesz kernels, and thereby certifying positivity as in \eqref{eq:cm0}.
Section \ref{sec5} brings commutative algebra into our tool box.
The relevant algebra arises from the convolution product on Borel measures.
We are interested in finitely generated subalgebras and
the polynomial ideals representing these. Elements in such
ideals can be interpreted as formulas for Riesz kernels.
In the setting of Section \ref{sec3} we recover the
 Orlik-Terao algebra \cite{OrlikTerao} of a hyperplane arrangement, 
 but now realized as a convolution algebra of piecewise polynomial
 volume functions.

In Section \ref{sec6} we give partial answers to questions raised
in the literature, namely in \cite[Conjecture 3.5]{MSUZ2016} and in
\cite[Conjecture 1.11]{scott2014}. Scott and Sokal considered 
elementary  symmetric polynomials $p$, and they conjectured
necessary and sufficient conditions on negative $s$ 
for complete monotonicity of $p^s$. We prove the necessity of their conditions. 
We also show that $p^s$ is completely monotone for sufficiently negative
exponents $s$, and we explain how to build the Riesz kernel in this setting.

In Section \ref{sec7}  we view the Riesz kernel $q$ of  $f = p^s$
as a function of the coefficients of the hyperbolic polynomial $p$.
In the setting of Section \ref{sec3}, when $f = \ell_1^{s_1} \cdots \ell_m^{s_m}$
for linear forms $\ell_i$,
we examine the dependence of $q$ on the coefficients of the $\ell_i$. 
These dependences are hypergeometric in nature.
In the former case, the Riesz kernel satisfies the
$\mathcal{A}$-hypergeometric system \cite[Section 3.1]{SST}, where
$\mathcal{A} $ is the support of $p$. In the latter case,
it is a Gel'fand-Aomoto hypergeometric function \cite[Section 1.5]{SST}.
This explains the computation in \cite[Example 3.11]{MSUZ2016}, and it opens
up the possibility to use D-modules as in~\cite[Section 5.3]{SST} or \cite{Wal} for 
further developing the strand of research initiated here.

\section{Complete Monotonicity}\label{sec2}

We now start afresh, by offering a more general definition of complete monotonicity.
Let $V$ denote a finite-dimensional real vector space and $C$ an open convex cone in $V$.
Its closed dual cone is denoted
$C^*\,=\,\{y\in V^*: \langle y,x\rangle \geq 0 \,\,\hbox{for all}\, \,x\in C\}$. Points in $C^*$
are linear functions that are nonnegative on $C$.
We are interested in differentiable functions on the open cone $C$
that satisfy the following very strong notion of positivity.

\begin{defn}[Complete monotonicity]
\label{def:cm}
A function $\,f\colon C\rightarrow \R\,$ is 
\emph{completely monotone} if $f$ is $\mathcal{C}^{\infty}$-differentiable and,
for all $k \in \mathbb{N}$ and
all vectors $v_1,\dots,v_k\in C$, we~have
\begin{align}\label{eq:cm}
  (-1)^kD_{v_1}\cdots D_{v_k} f(x) \,\,\geq \,\, 0 \quad \hbox{   for all $x\in C$.}
\end{align}
Here $D_v$ is the directional derivative along the vector $v$.
In this definition  it suffices to assume that $v_1,\ldots,v_k$ are extreme rays   
of $C$. This is relevant when $C$ is polyhedral.
\end{defn}

\begin{exmp}\label{ex:cm}
For $V=\R^n$ and $C=\R_{>0}^n$, the positive orthant, \eqref{eq:cm} is equivalent to~\eqref{eq:cm0}.
Moreover, if $C \supset \R_{>0}^n$ holds in Definition~\ref{def:cm}
then $f$ is completely monotone on $\R_{>0}^n$~too.
\end{exmp}

Since differentiation commutes with specialization, we have the following remark.

\begin{rem} \label{rem:restriction}
Let $W$ be a linear subspace of $V$ such that
$C \cap W$ is nonempty. If $f$ is completely monotone on $C$,
then the restriction $f|_W$ is completely monotone on~$C \cap W$.
\end{rem}

No polynomial of positive degree is completely monotone.
Indeed, by Remark~\ref{rem:restriction}, it suffices  to consider the case
$n=1$ and $C = \R_{>0}$.
Derivatives of univariate polynomials cannot alternate their sign for $x \gg 0$.
But, this is easily possible for rational functions.

\begin{exmp} \label{ex:mono}
Fix arbitrary negative real numbers  $s_1,s_2,\ldots,s_n$. Then the function
$$ f \,\,=\,\, x_1^{s_1}  x_2^{s_2} \cdots   x_n^{s_n}  $$
is completely monotone on $C = \R_{>0}^n$.
Remark \ref{rem:restriction} says that we may  replace the $x_i$ by linear forms.
We conclude that products of negative powers of linear forms are
completely monotone on their polyhedral cones.
This is the theme we study in Section~\ref{sec3}.
\end{exmp}

The Bernstein-Hausdorff-Widder Theorem \cite[Thm.~$12a$]{Widder1972} characterizes completely monotone functions in one variable. They are obtained as  Laplace transforms of Borel
measures on the positive reals.  This result was generalized 
by Choquet \cite[Thm.~$10$]{Choquet69}.
He proved the following theorem for higher dimensions and
arbitrary convex cones $C$.

\begin{thm}[Bernstein-Hausdorff-Widder-Choquet theorem]\label{thm:BHWC}
Let $f$ be a real-valued function on the open cone $C$. Then $f$ is completely monotone if and only if it is the Laplace transform of a unique Borel measure $\mu$ supported on the dual cone $C^*$, that~is,
\begin{align}\label{eq:Laplace}
  f(x) \,\,= \, \int_{C^*}e^{- \langle y,x \rangle }\, d\mu (y).
\end{align}
\end{thm}
\begin{rem}
By the formula \eqref{eq:Laplace}, any completely monotone function $f: C\rightarrow \R$ extends to a holomorphic function on the complex tube $C+i \cdot V=\{x+i\cdot v: x\in C, v\in V\}\subset V_{\C}$.
\end{rem}

The if direction in Theorem~\ref{thm:BHWC} is seen by
differentiating under the integral sign. Indeed, the left hand side of \eqref{eq:cm}
equals the following function which is nonnegative:
$$ (-1)^kD_{v_1}\dots D_{v_k} f(x) \,\, =\,\,
\int_{C^*} e^{- \langle y,x \rangle } \prod_{j=1}^k \langle y,v_i \rangle\, d\mu (y) .$$
The integrand is a positive real number 
whenever $x,v_1,\ldots,v_k \in C$ and $y \in C^*$.
The more difficult part is the only-if direction. Here one needs to construct the
integral representation. This is precisely our topic of study in this paper, for certain functions~$f$.

The Borel measure $\mu$ in the integral representation (\ref{eq:Laplace})
is called the {\em Riesz measure} when $f = p^s$. As we shall see, its support
 can be a lower-dimensional subset of $C^*$.
But in many  cases the Riesz measure is absolutely continuous 
with respect to Lesbesgue measure on $C^*$, that is, there exists a measurable function $q: C^* \rightarrow \R_{\geq 0}$ such that
\begin{align}
  d\mu(y) \,\, =\,\, q(y) dy.
  \end{align}
The nonnegative function $q(y)$ is called the {\em Riesz kernel} of $f=p^s$.
It serves as the certificate for complete monotonicity of $f$,
and our aim is to derive formulas for $q(y)$.

We next present a formula for the Riesz kernel in the situation of Example~\ref{ex:mono}.

\begin{prop}\label{prop:LapTrProdxi}
For any positive real numbers  $\alpha_1,\dots,\alpha_n$, we have  
  \begin{equation} \label{eq:product}
    x_1^{-\alpha_1}\cdots x_n^{-\alpha_n} \,=\, \int_{\R^n_{\geq 0}} e^{-\langle y,x \rangle}
    \frac{y_1^{\alpha_1-1}\cdots y_n^{\alpha_n-1}}{\Gamma(\alpha_1)\cdots\Gamma(\alpha_n)} \,dy
\qquad    \hbox{for all}\,\,\, x\in \R^n_{>0}.
  \end{equation}
  Hence the Riesz kernel of $\,f = \prod_{i=1}^n x_i^{-\alpha_i} \,$ is equal to
$ \, q(y) = \prod_{i=1}^n \Gamma(\alpha_i)^{-1} y_i^{\alpha_i-1} $.
\end{prop}

\begin{proof}
Recall the definition of the gamma function from its
integral representation:
$$ \Gamma(a) \,\,\, =\,\,\, \int_{\R_{\geq 0}} t^{a-1} e^{-t} dt. $$
Regarding $x > 0$ as a constant, we perform the change of variables $t = yx$.
This gives
 \begin{equation} \label{eq:product1}
x^{-a}\,\,\, =\,\,\,  \frac{1}{\Gamma(a)} \int_{\R_{\geq 0}} y^{a-1} e^{-yx} dy. 
\end{equation}
This is the $n=1$ case of the desired formula \eqref{eq:product}. The general case $n>1$ is
obtained by multiplying $n$ distinct copies of the univariate formula \eqref{eq:product1}.
\end{proof}

The Riesz kernel offers a certificate for a function $f$ to be completely monotone.
This is very powerful because it proves that infinitely many
functions are nonnegative on $C$.

\begin{exmp} \label{ex:intro}
Let $V = \R^3$. Note that $p = x_1x_2 + x_1 x_3 +x_2x_3$ is
positive on $C=\R^3_{>0}$. We consider negative powers
$f = p^{-\alpha}$ of this polynomial. The function $f$ is not completely monotone 
for $\alpha < 1/2$. For instance, if $\alpha = 5/11$ then this is proved by
$$ \qquad \frac{\partial^{12} f}{\partial x_1^4 \,\partial x_2^4 \,\partial x_3^4}
\quad = \quad
p^{-137/11} \cdot \hbox{a polynomial of degree $12$ with $61$ terms.} $$
Namely, the value of this function at the positive point $(1,1,1)$ is found to be
$$ \begin{matrix} -\frac{16652440985600}{762638095543203}3^{6/11}
\,\, = \,\, -0.0397564287 \,\, < \,\, 0.  \end{matrix} $$

Suppose now that  $\alpha > 1/2$.
We claim that no such violation exists, i.e.~the function
  $f = p^{-\alpha}\,$ is completely monotone. 
  This is certified by
  exhibiting the Riesz kernel:
  $$q(y) \,\,\,=\,\,\,\frac{2^{1-\alpha}}{(2\pi)^{\frac{1}{2}}\Gamma(\alpha)\Gamma(\alpha-\frac{1}{2})}
  \left(y_1 y_2 + y_1 y_3 + y_2 y_3
  -\frac{1}{2}(y_1^2+y_2^2+y_3^2)\right)^{\! \alpha-\frac{3}{2}} .$$
  Then the integral representation \eqref{eq:Laplace}
  holds with the Borel measure $\, d\mu(y) = q(y) dy$. A proof can be found in
  \cite[Corollary 5.8]{scott2014}.   We invite our readers to verify this formula.
   \end{exmp}

We next discuss the implications of Remark \ref{rem:restriction}
for the measure $d\mu$ in \eqref{eq:Laplace}.
 Suppose that $f(x)=\int_{C^*}e^{- \langle y,x \rangle }\, d\mu (y)$ 
 holds for a cone $C$ in a vector space $V$.  We want to restrict $f$ 
 to a subspace $W$.  The inclusion $W\subset V$ induces a
  \emph{projection} $\pi:V^*\rightarrow W^*$,    by restricting the linear form.
   The dual of the cone $C\cap W$ equals $\pi(C^*)$ in
      the space $W^*$. We are looking for a representation 
   $f(x)=\int_{(C\cap W)^*}e^{- \langle y,x \rangle }\, d\mu' (y)$, $x\in C\cap W$.
    The measure $\mu'$ is obtained from 
     the measure $\mu$ by the general \emph{push-forward} construction.
          
\begin{defn}
Let $\mu$ be a measure on a set $S$ and $\pi :S\rightarrow S'$ a measurable function.
 The push-forward measure $\pi_*\mu$ on $S'$ is defined by setting 
 $\,(\pi_*\mu)(U):=\mu(\pi^{-1}(U))$, where $U\subset S^\prime$ is any measurable subset. 
\end{defn}

\begin{exmp}
Suppose we are given a (measurable) function $\pi:V_1\rightarrow V_2$ 
of real vector spaces and a measure with density $g$ on $V_1$. 
Then $(\pi_*\mu)(U)=\int_{\pi^{-1}(U)}g(x) dx$.
\end{exmp}

The following results are straightforward from the definitions.

\begin{prop}\label{prop:push-forward}
Let $f = p^s$ be completely monotone on a convex cone $C\subset V$ and $\mu$ the Riesz 
measure of $f$. Then the Riesz measure of the restriction of $f$ to the subspace $W$ is the push-forward $\pi_*\mu$, where $\pi:V^*\rightarrow W^*$ is the projection.
\end{prop}

\begin{cor}
A function $f$ is completely monotone on the cone $C \subset V$ if and only if, for every
$g \in {\rm GL}(V)$, the function $f \circ g^{-1}$ is completely monotone on $g(C)$.
The Riesz measures transform accordingly under the linear transformation dual to $g$.
\end{cor}

   


\section{Monomials in Linear Forms}\label{sec3}

Let $\ell_1,\ell_2,\ldots,\ell_m$ be linear forms in $n$ variables and consider the open polyhedral~cone
\begin{equation}
\label{eq:polyhedralcone}
 C \,\, = \,\, \bigl\{ v \in \RR^n \, : \, \ell_i(v) > 0 \,\,\,\hbox{for} \,\,\,
i=1,2,\ldots,m \bigr\}. 
\end{equation}
The dual cone $C^*$ is spanned by $\ell_1,\ell_2,\ldots,\ell_m$ in
$(\RR^n)^*$. A generator $\ell_i$ is an extreme ray of the cone $C^*$ if 
and only if the face $\{v \in C : \ell_i(v) = 0\}$ is a facet of $C$.

\begin{prop} \label{prop:pushdirichlet}
For any positive real numbers $\alpha_1,\alpha_2,\ldots,\alpha_m$, the function
$$ f \quad = \quad \ell_1^{-\alpha_1} \ell_2^{-\alpha_2} \cdots  \, \ell_m^{-\alpha_m} 
$$ 
is completely monotone on the cone $C$. The Riesz measure of $f$ is the
pushforward of the Dirichlet measure
$\,\frac{y_1^{\alpha_1-1} \cdots \,\,y_m^{\alpha_m-1}}{ \Gamma(\alpha_1)
\,\,\cdots \, \,\Gamma(\alpha_m)} dy\,$ on the orthant $\RR_{> 0}^m$ under the linear map 
$\,L:\RR_{\geq 0}^m \rightarrow C^*$ that takes the standard basis 
$e_1,\ldots,e_m$ to the linear forms $\ell_1,\ldots,\ell_m$.
\end{prop}

\begin{proof} The monomial $x_1^{-\alpha_1} \cdots x_m^{-\alpha_m}$ 
is completely monotone on $\R_{>0}^m$ by Example~\ref{ex:mono}.
Our function $f$ is its pullback under
the linear map $ C \rightarrow \RR_{> 0}^m$ that is dual to
the map $L: \RR_{\geq 0}^m \rightarrow C^*$. The result follows
from  a slight extension of Proposition \ref{prop:push-forward}.
In that proposition, the map $L$ would have been surjective but this is not really necessary.
\end{proof}

In what follows we assume that 
the dual cone $C^*$ is full-dimensional in $\RR^n$. 
Then $\ell_1,\ldots,\ell_m$ span $(\RR^n)^*$.
The case ${\rm dim}(C^*) < n$ is discussed at the end of this section.
There is a canonical polyhedral subdivision of $C^*$, called the
{\em chamber complex}. It is the common refinement of all
cones spanned by linearly independent subsets of $\{\ell_1,\ldots,\ell_m\}$.
Thus, for a general vector $\ell$ in $C^*$, the chamber containing $\ell$
is the intersection of all halfspaces defined by linear inequalities
 ${\rm det}(\ell_{i_1},\ldots,\ell_{i_{n-1}}, y) > 0$
that hold for $y = \ell$. The collection of all chambers and their faces
forms a polyhedral fan with support~$C^*$. That fan is the
chamber complex of  the map $L$, which we regard as an $n \times m$-matrix.

\begin{exmp} \label{ex:3by5}
Let $n=3$, $m=5$ and consider the matrix
\begin{equation}
\label{eq:3by5} L \quad = \quad 
\begin{small} \begin{pmatrix}
1 & 1 & 1 & 1 & 1 \\
0 & 1 & 2 & 1 & 0 \\
0 & 0 & 1 &  2 & 1  \end{pmatrix} . \end{small}
\end{equation}
The corresponding linear forms $\ell_1,\ldots,\ell_5$
are the entries of the row vector $ (x_1, x_2, x_3) L $.
For any positive real numbers $\alpha_1,\alpha_2,\alpha_3,\alpha_4,\alpha_5$,
 we are interested in the function
$$ f \,\,=\,\, x_1^{-\alpha_1} (x_1 + x_2)^{-\alpha_2} (x_1 + 2 x_2 + x_3)^{-\alpha_3}
(x_1 + x_2 + 2 x_3)^{-\alpha_4} (x_1 + x_3)^{-\alpha_5}.
$$
This is completely monotone on the pentagonal cone $C$,
consisting of all vectors $(x_1,x_2,x_3)$ where the five linear forms are positive.
Its dual cone $C^*$ is also a pentagonal cone. The elements $(y_1,y_2,y_3)$ of $C^*$ are 
nonnegative linear combinations of the columns of~$L$.

The chamber complex of $L$ consists of $11$ full-dimensional cones.
Ten of them are triangular cones. The remaining cone, right in the middle of $C^*$,
is pentagonal. It~equals
\begin{equation}
\label{eq:centralchamber} \bigl\{\, y \in \RR^3 \,\,:\,\,
 y_1 \geq  y_2 ,\,\,
 y_1 \geq  y_3  ,\,\,
2 y_3 \geq y_2, \,\,
2 y_2 \geq y_3 \,\,{\rm and}\,\,
y_2+y_3 \geq y_1 \,\bigr\}. 
\end{equation}
\end{exmp}  

Returning to the general case, 
for each $y$ in the interior of $C^*$, the  fiber $L^{-1}(y)$ of the
map  $L: \RR_{\geq 0}^m \rightarrow C^*$
is a convex polytope of dimension $m-n$. The polytope $L^{-1}(y)$
is simple provided $y$ lies in an open cone of the chamber complex.
All fibers $L^{-1}(y)$ lie in affine spaces that are parallel to ${\rm kernel}(L) \simeq \RR^{m-n}$.
In the following we denote by $\vert L\vert$  the square root of
the \emph{Gram determinant} of a matrix $L$, that is, $\vert L\vert = \sqrt{\det (LL^T)}$.  

\begin{thm}\label{thm:prodlin}
The Riesz kernel $q(y)$ of $ f = \prod_{i=1}^m \ell_i^{-\alpha_i}$ is a continuous function on the polyhedral cone $C^*$ that is homogeneous of degree
$\sum_{i=1}^m \alpha_i - n$. Its value $q(y)$ at a point $y \in C^*$ equals the integral of 
$\,|L|^{-1} \prod_{i=1}^m y_i^{\alpha_i-1} \Gamma(\alpha_i)^{-1}\,$ over the fiber $L^{-1}(y)$ with respect to $(m-n)$-dimensional Lebesgue measure. If $\alpha_1,\ldots,\alpha_m$ are integers
then the Riesz kernel $q(y)$ is piecewise polynomial and differentiable of order
$\sum_{i=1}^m \alpha_i - n-1$. To be precise, $q(y)$ is a homogeneous polynomial 
   on each cone in the chamber complex.
\end{thm}

\begin{proof}
We first consider the case $\alpha_1 = \cdots = \alpha_m = 1$.
Here Proposition  \ref{prop:LapTrProdxi} says that
  \begin{equation} \label{eq:product2}
    x_1^{-1}\cdots x_m^{-1} \,=\, \int_{\R^m_{\geq 0}} e^{-\langle y,x \rangle} \,dy.
  \end{equation}
The Riesz measure of this function is Lebesgue measure on $\RR_{\geq 0}^m$,
and its Riesz kernel is the constant function $1$. The pushforward of this
measure under the linear map $L$ is absolutely continuous with respect to
Lebesgue measure on $C^*$. The Riesz kernel 
of this pushforward is the function $q(y)$ whose value at $y\in C^*$ is the volume of the fiber
$L^{-1}(y)$ divided by $|L|$. This volume is a function in $y$ 
that is homogeneous of degree $m-n $.
Moreover, the  volume function is a polynomial on each chamber and it is 
differentiable of order $m-n-1$. This is known from the
theory of splines; see e.g.~the book by De Concini and  Procesi \cite{DP}.
This proves the claim since $ m= \sum_{i=1}^m \alpha_i $.

Next we let $\alpha_1,\ldots,\alpha_m$ be arbitrary positive integers.
We  form a new matrix $L^{(\alpha)}$ from $L$
by replicating $\alpha_i$ identical columns from the column  $\ell_i$ of $L$.
Thus the number of columns of $L^{(\alpha)}$ is $m' = \sum_{i=1}^m \alpha_i$.
The matrix $L^{(\alpha)}$ represents a linear map from $\RR_{\geq 0}^{m'}$
onto the same cone $C^*$ as before, and also the chamber complex
of $L^{(\alpha)}$ remains the same.
However, its fibers are polytopes of 
dimension $m'-n$, which is generally larger than $m-n$.
We now  apply the argument in the previous paragraph to this map. This gives the asserted result,
namely the Riesz kernel $q(y)$ is piecewise polynomial of degree $m'-n$, and is
differentiable of order $m'-n-1$ across walls of the chamber complex.

We finally consider the general case when the $\alpha_i$ are arbitrary 
positive real numbers. The formula for the Riesz kernel follows by 
Proposition \ref{prop:pushdirichlet}. The homogeneity and continuity properties 
of the integral over $\,x_1^{\alpha_1 - 1} \cdots x_m^{\alpha_m-1}\,$
hold here as well, because this function is well-defined and
positive on each fiber $L^{-1}(y)$.
\end{proof}

\begin{exmp} \label{ex:24}
Let $n=2$, $m=4$ and 
$L = \begin{pmatrix} 3 & 2 & 1 & 0 \\ 0 & 1 & 2 & 3 \end{pmatrix}$.
Here $C^* = \RR_{\geq 0}^2$ and the chamber complex is the
division into three cones defined by the lines
$y_2= 2y_1$ and $y_1 = 2y_2$. The  corresponding monomial in linear forms 
has the integral representation
$$
(3x_1)^{-\alpha_1} (2 x_1 + x_2)^{-\alpha_2} (x_1 + 2x_2)^{-\alpha_3}
(3 x_2)^{-\alpha_4}
\quad = \quad
\int_{\R_{\geq 0}^2} \!\! e^{-y_1 x_1 - y_2 x_2} q(y_1,y_2) d y_1 d y_2.
$$
For $\alpha_1 = \alpha_2 = \alpha_3 = \alpha_4 = 1$, the Riesz kernel is
the piecewise quadratic function
$$ q(y_1,y_2) \,\,\,= \,\,\,
\frac{1}{108}\cdot \begin{cases}
\qquad 3 y_1^2 &  \hbox{if $ y_2 \geq 2 y_1$,} \\
4 y_1 y_2 - y_1^2 - y_2^2 & \hbox{if $ y_2 \leq 2 y_1$ and $y_1 \leq 2 y_2$,} \\
\qquad 3 y_2^2 & \hbox{if $ y_1 \geq 2 y_2 $.}
\end{cases}
$$
The three binary quadrics measure the areas of the convex
polygons $L^{-1}(y)$. These polygons are triangles, quadrilaterals and triangles,
in the three cases. Note that $q(y)$ is differentiable.
For small values of $k \geq 1$, it is instructive to work out the piecewise polynomials $q(y)$ of degree $2+k$ for all $\binom{3+k}{3}$ cases when $\alpha_i \in \mathbb{N}_+$
with $\sum_{i=1}^4 \alpha_i = 4+k$.
\end{exmp}

We have shown that the reciprocal product of linear forms is
the Laplace transform of the piecewise polynomial function $q(y)$ which
measures the volumes of fibers of $L$:
$$ \ell_1^{-1}(x) \cdots \ell_m^{-1}(x) \quad = \quad
\int_{C^*} e^{\langle y ,x\rangle} q(y) dy . $$
The volume function can be recovered from this formula
by applying the inverse Laplace transform. This technique
was applied by  Lasserre and Zerron \cite{LZ1,LZ2} to develop
a surprisingly efficient algorithm for computing the volumes of convex polytopes.

\begin{exmp} \label{ex:35}
Fix the matrix $L$ in Example~\ref{ex:3by5}, and consider any vector
$y$ in the central chamber  \eqref{eq:centralchamber}. The fiber
$L^{-1}(y)$ is a pentagon. Its vertices are the rows~of
$$
\begin{small}
\frac{1}{6} 
\begin{pmatrix}
0 &  \! 3 y_1+3 y_2-3 y_3 &  0 &  \!\! -3 y_1+3 y_2+3 y_3 \! &  6 y_1-6 y_2 \\
0 &  6 y_1-6 y_3 & \!\!  -3 y_1+3 y_2+3 y_3 \! &  0 &  \! 3 y_1-3 y_2+3 y_3 \! \\ 
 6 y_1-6 y_3 &  0 &  3 y_2 &  0 &  -3 y_2+6 y_3 \\
 6 y_1-2 y_2-2 y_3 \! &  0 &  4 y_2-2 y_3 &  -2 y_2+4 y_3 &  0 \\  
 6 y_1-6 y_2 &  6 y_2-3 y_3 &  0 &  3 y_3 &  0
\end{pmatrix}.
\end{small}
$$
The Riesz kernel is given by the area of this pentagon:
$$ \begin{matrix} q(y) \,\, = \,\, \frac{1}{6}{\rm area}(L^{-1}(y)) \,\,=\,
\frac{1}{24}(6y_1y_2 + 6 y_1 y_3 + 2 y_2 y_3 -3 y_1^2 -5 y_2^2 -5 y_3^2).
\end{matrix} $$
Similar quadratic formulas hold for the Riesz kernel
 $q(y)$ on the other ten chambers.
\end{exmp}

\begin{rem}\label{rem:degenerate}
Proposition \ref{prop:pushdirichlet} also applies in the case
when $\ell_1,\ldots,\ell_m$ do not span $\RR^n$. Here
the cone $C^*$ is lower-dimensional. The Riesz measure is
supported on that cone. There is no Riesz kernel $q(y)$ in the
sense above. But, we can still consider
the volume function on the cone~$C^*$.
If $m=1$, then $f = \ell^{-\alpha}$ is a negative power of
a single linear form. This function $f$ is completely monotone on the
half-space $C = \{\ell > 0\}$. The cone
$C^*=\{t\ell:t \geq 0\}$ is the ray spanned by $\ell$. The
Riesz measure is absolutely continuous with respect to Lebesgue measure $dt$ on that ray,
with density $t^{\alpha-1}/\Gamma(\alpha)$, cf. \eqref{eq:product1}.
\end{rem}

\section{Hyperbolic Polynomials}\label{sec4}

We now return to the setting of Section \ref{sec2}
where we considered negative powers $f = p^{-\alpha}$
of a homogeneous polynomial. We saw in Section \ref{sec3}
that this is completely monotone if $p$ is a product of linear forms.
For which polynomials $p$ and which $\alpha$ 
can we expect such a function to be completely monotone?
We begin with a key example.

Fix an integer $m \geq 0$, set $n = \binom{m+1}{2}$,
and let  $x = (x_{ij}) $ be a symmetric $m \times m$ matrix of unknowns.
Its determinant ${\rm det}(x)$ is a homogeneous polynomial of degree~$m$
in the $x_{ij}$. We write $C$ for the open cone in $\RR^n$ consisting of all
positive definite matrices. Up to closure, this cone is self-dual with respect to the
trace inner product $\langle x,y\rangle = {\rm trace}(xy)$. Thus,
 $C^*$ is the closed cone of positive
semidefinite (psd) matrices in $\RR^n$.

\begin{thm}[Scott-Sokal \cite{scott2014}] \label{thm:det}
  Let $\alpha\in \R$. The function $\det(x)^{-\alpha}$ is completely monotone on the cone 
  $C$  of positive definite symmetric $m \times m$ matrices if and only if 
  $$ \begin{matrix}\,\alpha \in \{0, \frac{1}{2}, 1, \frac{3}{2},\dots, \frac{m-1}{2}\} \quad
  \hbox{or} \quad  \alpha> \frac{m-1}{2}. \end{matrix} $$
  \end{thm}

\begin{proof}
The if-direction of this theorem is a classical result in statistics, as we explain below. 
We define a probability distribution $N$ on the space of $m\times r$ matrices as follows: 
each column is chosen independently at random with respect to the $m$-variate normal distribution with mean zero and covariance matrix $\Sigma$. The \emph{Wishart distribution} $W_m(\Sigma,r)$ is the push-forward of the distribution $N$ to the space of symmetric $m\times m$ matrices, by the map 
$Z \rightarrow ZZ^T$. Since $ZZ^T$ is always psd,  the Wishart distribution is 
supported on the closure of $C$. There are two main cases one needs to distinguish. 

If $r\geq m$, then  the matrix $ZZ^T$ has full rank $m$ with probability one.
Hence, the support of the Wishart distribution coincides with $C$.
An explicit formula for the Riesz measure can be derived from the Wishart distribution.
This is explained in detail in \cite[Proposition 3.8]{MSUZ2016}. 
Setting $\alpha = r/2$, the Riesz kernel for $\det(x)^{-\alpha}$ is equal to
\begin{equation}
\label{eq:psdriesz}
\begin{matrix}
  q(y) \,\, = \,\,\biggl(\pi^{\frac{m(m-1)}{4}} \prod\limits_{j=0}^{m-1} \Gamma\left( \alpha- \frac{j}{2}\right)
  \biggr)^{\! -1} \!\! \det(y)^{\alpha-\frac{m+1}{2}}.
  \end{matrix}
\end{equation}
This formula holds whenever $\alpha>(m-1)/2$.
It proves that $\det(x)^{-\alpha}$ is completely monotone on $C$ for this range of $\alpha$.

If $r<m$ then the matrix $ZZ^T$ has rank at most $r$.
Hence the Wishart distribution is supported on the 
subset of $m \times m$ psd matrices of rank at most $r$.
  We have
$$ 1 \,\,=\,\, \int_{C^*} W_m(\Sigma,r)
\,\,\, =\,\,\left((2\pi)^m\det\Sigma\right)^{-r/2}\int_{(\RR^m)^r} \!\! \exp\biggl({-\frac{1}{2}\sum_{i=1}^r z_i^T \Sigma^{-1}z_i}\biggr)
 dz_1\cdots dz_r.$$
Let $Q :=(2\Sigma)^{-1}$. Multiplying both sides in the formula above by $(\det Q)^{-r/2}$, we obtain
$$(\det Q)^{-r/2}\,\,\,=\,\,\,\int_{(\RR^m)^r}\!\! \exp\biggl({-\sum_{i=1}^r z_i^T  Q z_i} \biggr) \pi^{-mr/2}\, dz_1\cdots dz_r.$$
Here $z_i$ are the columns of an $m\times r$ matrix $Z$. Hence we may rewrite this formula as
$$(\det Q)^{-r/2} \,
\,\,=\,\,\int_{\RR^{m\times r}}\!\! \exp\left({- \tr (Q ZZ^T)}\right) \pi^{-mr/2} \,dZ.$$
This establishes complete monotonicity of $\det(x)^{-r/2}$, $x\in C$, for $r=0,1,\dots,m{-}1$.
Indeed, we realized the Riesz measure on $C^*$ as the push-forward of the (scaled) Lebesgue measure $dZ/\pi^{mr/2}$ on the space $\R^{m\times r}$ of $m\times r$ matrices under the map $Z\mapsto ZZ^T$.
The only-if direction of Theorem \ref{thm:det}
is due to Scott and Sokal 
\cite[Theorem 1.3]{scott2014}.
\end{proof}

According to Remark \ref{rem:restriction},
the restriction of a  completely monotone function to a linear subspace 
 gives a completely monotone function. We thus obtain
 completely monotone functions by restricting the determinant function to
 linear spaces of symmetric matrices. This can be constructed as follows.
Fix linearly independent
symmetric $m \times m$ matrices $A_1,\ldots,A_n $ such that the following 
{\em spectrahedral cone}
is non-empty:
$$ C \quad = \quad \bigl\{ \,x \in \RR^n \,:\, x_1 A_1+\dots+ x_n A_n \,\,\,
\hbox{is positive definite} \bigr\}. $$
Its dual $C^*$ is the image of the
cone of $m \times m$ psd matrices under
the linear map $L$ that is dual to the inclusion $x\in \R^n\mapsto x_1A_1+\dots+x_nA_n$.
Such a cone is known as a {\em spectrahedral shadow}.
 The following  polynomial  vanishes
 on the boundary of~$C$:
 \begin{equation}
 \label{eq:pAAA}
 p(x)\,\, = \,\, {\rm det}(x_1 A_1+\dots+ x_n A_n) .
 \end{equation}

\begin{cor} \label{cor:A1An}
Let $\,\alpha \in \{0, \frac{1}{2}, 
 1, \frac{3}{2},\dots, \frac{m-1}{2}\} \,$ or $\,    \alpha> \frac{m-1}{2}$.
 Then the function $ p^{-\alpha}$ is completely monotone on the spectrahedral cone $C$ and its Riesz measure is the push-forward of the
Riesz measure from the proof of Theorem \ref{thm:det} under the map $L$.
\end{cor}

Consider the  case when $A_1,\ldots,A_n$ are diagonal matrices.
Here the polynomial $p(x)$ is a product of linear forms, and we are 
in the situation of Section \ref{sec3}. For general symmetric matrices
$A_i$, the fibers $L^{-1}(y)$ are spectrahedra and not polytopes.
If $ \alpha> \frac{m-1}{2}$ then the Riesz kernel exists, and its values are
found by integrating \eqref{eq:psdriesz} over the spectrahedra $L^{-1}(y)$.
In particular, if $\alpha = \frac{m+1}{2}$, then the 
value of the Riesz kernel $q(y)$ equals, up to a constant, the volume of
the spectrahedron $L^{-1}(y)$. The role of the chamber complex
is now played by a nonlinear branch locus. It would be 
desirable to compute explicit formulas, in the spirit of
 Examples \ref{ex:24} and~\ref{ex:35}, for these spectrahedal volume functions.

\begin{exmp}
Let $m=n=3$ and write $C$ for the cone of positive definite matrices 
\begin{equation}
\label{eq:uhlerex}
\begin{pmatrix}
x_1+x_2+x_3 & x_3 & x_2 \\
x_3 & x_1+x_2+x_3 & x_1 \\
x_2 & x_1 & x_1+x_2+x_3
\end{pmatrix}.
\end{equation}
This space of matrices is featured in \cite[Example 1.1]{SU} where it serves the
prominent role of illustrating the convex algebraic geometry of Gaussians.
Let $p(x)$ be the determinant of \eqref{eq:uhlerex}. Cross sections of
the cone $C$ and its dual $C^*$ are shown in the middle and right of \cite[Figure 1]{SU}.
The fibers $L^{-1}(y)$ are the $3$-dimensional spectrahedra shown on the
left in \cite[Figure 1]{SU}. Their volumes are computed by 
the Riesz kernel $q(y)$ for $\alpha=2$.
\end{exmp}

We have seen that determinants of symmetric matrices of linear forms are a natural
class of polynomials $p(x)$ admitting negative powers that are completely monotone.
Which other polynomials have this property? 
The section title reveals the~answer.

\begin{defn}
A  homogeneous polynomial $p\in \R[x_1,\dots,x_n]$ is \emph{hyperbolic} for
a vector $e\in \R^n$ if $p(e)>0$ and, for any $x\in \R^n$,  the univariate polynomial
 $t\mapsto p(t\cdot e-x)$ has only real zeros.  Let $C$ be the connected component of the set
 $\R^n\setminus \{p=0\}$  that contains $e$.
If $p$ is hyperbolic for $e$, then it is hyperbolic for all vectors in $C$.
In that case, $C$  is an open convex cone, called the \emph{hyperbolicity cone} of $p$.
 Equivalently, a homogeneous polynomial $p\in \R[x_1,\dots, x_n]$ is hyperbolic with
  hyperbolicity cone $C$ if and only if $p(z)\neq 0$ for any 
  vector $z$ in the tube domain $C+i\cdot\R^n$ in the complex space $ \C^n$.
  \end{defn}

\begin{exmp}\label{ex:det}
The canonical example of a hyperbolic polynomial is the determinant $p(x) = \det(x)$ of a
symmetric $m\times m$ matrix $x$. Here we take $e$ to be the identity matrix.
Then $C$ is the cone of positive definite symmetric matrices. Hyperbolicity holds because the
roots of $p(t\cdot e-x)$ are the eigenvalues of $x$, and these are all real.
\end{exmp}

\begin{exmp} \label{ex:resthyp}
The restriction of a hyperbolic polynomial $p$ to a linear subspace $V$ is hyperbolic,
provided $V$ intersects the hyperbolicity cone $C$ of $p$.  Therefore the
polynomials in  \eqref{eq:pAAA} are all hyperbolic. In particular,
any product of linear forms $\ell_i$  is hyperbolic, if we take the hyperbolicity cone to be
the polyhedral cone $C$  in \eqref{eq:polyhedralcone}.
\end{exmp}

The following result states that  hyperbolic polynomials 
are precisely the relevant class of polynomials for our study of 
positivity certificates via integral representations.

\begin{thm}\label{thm:cm}
Let $p\in \R[x_1,\dots,x_n]$ be a homogeneous polynomial that is positive 
 on an open convex cone $C$ in $\RR^n$, and such that the power $f = p^{-\alpha}$ 
 is completely monotone on $C$ for some $\alpha> 0$.
 Then $p$ is hyperbolic and its hyperbolicity cone contains $C$.
 \end{thm}

\begin{proof}
This follows from \cite[Corollary 2.3]{scott2014} as explained in \cite[Theorem 3.3]{MSUZ2016}.
\end{proof}

A hyperbolic polynomial $p$ with hyperbolicity cone $C \subset \RR^n$ is called
\emph{complete} if $C$ is pointed, that is, the dual cone $C^*$ is $n$-dimensional.
 In the polyhedral setting of Proposition \ref{prop:pushdirichlet},
 this was precisely the condition for the Riesz kernel to exist. This fact generalizes.
The following important result due to G{\aa}rding \cite[Theorem 3.1]{garding1951}
should  give a complex integral representation of Riesz kernels for 
arbitrary hyperbolic polynomials.

\begin{thm}[G{\aa}rding]\label{thm:Riesz}
Let $p$ be a complete hyperbolic polynomial with hyperbolicity cone $C \subset \RR^n$.
Fix a vector $e \in C$ and  a complex number $\alpha$ with ${\rm Re}(\alpha)>n$, and define
  \begin{align}\label{eq:Riesz}
  \qquad  q_{\alpha}(y) \,\,=\,\, (2\pi)^{-n} \int_{\R^n} p(e+i\cdot x)^{-\alpha} 
  e^{\langle y , e+i\cdot x \rangle}\,dx
     \quad   \qquad \hbox{for}\,\,\, y\in \R^n,
  \end{align}
where $p(e+i\cdot x)^{-\alpha} = e^{-\alpha(\log|p(e+i\cdot x)|+i\arg p(e+i\cdot x))}$.
  Then $q_{\alpha}(y)$ is independent of the choice of $e\in C$ and vanishes when $y\notin C^*$.
If  $\,{\rm Re}(\alpha)>n+k$ then $q_\alpha(y)$ has continuous derivatives of order $k$ as a function of $y$, and is analytic in $\alpha$ for fixed $y$.   Moreover, 
 \begin{equation}
 \label{eq:Riesz1}
  p(x)^{-\alpha} \,\,= \,\,\int_{C^*} e^{-\langle y,x \rangle} q_\alpha(y)\,dy \qquad \quad
  \hbox{for}\,\,\, x\in C. \qquad
 \end{equation}
 \end{thm}

The function $q_\alpha(y)$ looks like a Riesz kernel for $f = p^{-\alpha}$.
However, it might lack one crucial property: we do not yet
know whether $q_\alpha(y) $ is nonnegative for $y \in C^*$.
If this holds then $f = p^{-\alpha}$ is completely monotone,
\eqref{eq:Riesz} gives a formula for the Riesz kernel of $f$, and
\eqref{eq:Riesz1} is the integral representation of $f$ that is promised
in Theorem~\ref{thm:BHWC}.

\begin{rem}\label{rem:Riesz}
The condition ${\rm Re}(\alpha)>n$ is only sufficient for $q_\alpha$
to be a well-defined function. However, for any hyperbolic  $p$
and any $\alpha \in \C$, the formula \eqref{eq:Riesz} defines a distribution on $C^*$;
see \cite[Section 4]{ABG1973}. If nonnegative, then this is the Riesz measure.
\end{rem}
 
 The following conjecture  is important for 
statistical applications of the models in~\cite{MSUZ2016}.

\begin{conj}[Conjecture $3.5$ in \cite{MSUZ2016}]\label{conj:MSUZ}
Let $p\in \R[x_1,\dots,x_n]$ be a complete hyperbolic polynomial with hyperbolicity cone $C$. Then there exists a real $\alpha >n$ such that $q_{\alpha}(y)$ in \eqref{eq:Riesz} is nonnegative on 
the dual cone $C^*$. In particular, $q_\alpha(y)$ is a Riesz kernel.
\end{conj}

Conjecture~\ref{conj:MSUZ} holds true for all hyperbolic polynomials
that admit a symmetric determinantal representation, by Corollary \ref{cor:A1An}.
This raises the question whether such a representation exists for every hyperbolic 
polynomial. The answer is negative.

\begin{exmp}[$n=4$]Consider the {\em specialized V\'amos polynomial} in
\cite[Section~3]{Kum}: 
$$ p(x) \,\,=\,\,
x_1^2 x_2^2 \,+\, 4 (x_1 + x_2 + x_3 + x_4)( x_1x_2x_3 + 
x_1x_2x_4 + x_1x_3x_4 + x_2x_3x_4 ). $$
It is known that no power of $p(x)$ admits a symmetric determinantal representation.
\end{exmp}


\begin{exmp}\label{ex:elem}
For any $m=0,1,\dots,n$ the $m$th elementary symmetric polynomial $E_{m,n}(x) = \sum_{1\leq i_1<\dots<i_m\leq n} x_{i_1}\cdots x_{i_m}$ is hyperbolic with respect to $e=(1,\dots,1)\in \R^n$.
The hyperbolicity cone of $E_{m,n}$  contains the orthant $\R^n_{>0}$. It is known that
$E_{m,n}$ has no symmetric determinantal representation when
$2 \leq m \leq n-2$, see \cite[Example 5.10]{KPV} for a proof.
Exponential varieties associated
with $E_{m,n}$ are studied in \cite[Section~6]{MSUZ2016}.
\end{exmp}

In Section \ref{sec6} we 
prove Conjecture \ref{conj:MSUZ} for all
elementary symmetric polynomials.
This is  nontrivial because 
Corollary \ref{cor:A1An} does not apply to $E_{m,n}$
when $2 \leq m \leq n-2$.
Scott and Sokal \cite{scott2014} gave a conjectural description
of the   set of parameters~$\alpha>0$ for which the function $E_{m,n}^{-\alpha}$
 is completely monotone on $\R_{>0}^n$. For a warm-up see Example~\ref{ex:intro}.

\begin{conj}[Conjecture $1.11$ in \cite{scott2014}]\label{conj:SS}
Let $2 \leq m \leq n$. Then $E_{m,n}^{-\alpha}$ is completely monotone on the 
positive orthant $\R_{>0}^n$ if and only if $\alpha= 0$ or $\alpha \geq (n-m)/2$.
\end{conj}

Scott and Sokal settled Conjecture \ref{conj:SS} for $m=2$. However,
they ``have been unable to find a proof of either the necessity or the sufficiency'' in general, as they stated in \cite[page~334]{scott2014}.
In Section \ref{sec6} we prove the only if direction of Conjecture~\ref{conj:SS}.

\begin{rem} 
The positive orthant $\R_{>0}^n$ is strictly contained in the hyperbolicity cone of
the elementary symmetric polynomial $E_{m,n}$ unless $m=n$. The next proposition ensures that we can replace $\R_{>0}^n$ by the
hyperbolicity cone in the formulation of Conjecture \ref{conj:SS}.
\end{rem}

\begin{prop}\label{prop:subcone}
Let $p$ be a hyperbolic polynomial with hyperbolicity cone $C$ and let 
$K \subset C$ be an open convex subcone. Then, for any fixed $\alpha>0$,
 the function $p^{-\alpha}$ is completely monotone on $C$ if and only if it is completely monotone on the subcone~$K$.
\end{prop}

\begin{proof}
The only if direction is obvious. For the if direction, we argue as follows.
For any $\alpha$, by results in \cite[Section 2]{ABG1973}, the function $p^{-\alpha}:C\rightarrow \R_{>0}$ is the Laplace transform of a distribution supported on $C^*$.
For $\alpha>n$ this already follows from Theorem \ref{thm:Riesz}. If $p^{-\alpha}$ is completely monotone on an open subcone $K\subset C$, then, by Theorem \ref{thm:BHWC}, it is the Laplace transform of a unique Borel measure supported on $K^*$. But since $C^*\subset K^*$ and since the Laplace transform on the space of distributions supported on $K^*$ is injective 
\cite[Proposition 6]{Schwartz}, the above implies that this Borel measure is supported on $C^*$.
We thus conclude that 
$p^{-\alpha}$ is completely monotone on $C$ by Theorem \ref{thm:BHWC}.
\end{proof}

\section{Convolution Algebras}\label{sec5}

In this section we examine the convolution of measures supported in a cone.
This  is a commutative product. It is mapped to multiplication of functions under the
Laplace transform. We obtain isomorphisms of commutative algebras
between convolution algebras of Riesz kernels and  algebras of functions they represent.
This allows to derive relations among Riesz kernels of completely monotone functions.
In the setting of Section~\ref{sec3}, we obtain a realization of the {\em Orlik-Terao algebra} \cite{OrlikTerao}
as a convolution algebra.

As before, we fix an open convex cone $C$ in $\RR^n$. We write
$\mathcal{M}_+(C)$ for the set of locally compact Borel measures 
that are supported in the closed dual cone $C^*\subset (\R^n)^*$.
These hypotheses on the measures $\mu \in \mathcal{M}_+(C)$ stipulate that
$\,0 \leq \mu(K)<+\infty\,$ for any compact set $\,K\subset (\R^n)^*$, and
$\,\mu(E)=0\,$ for any Borel set $\,E\subset (\R^n)^* \backslash C^*$.
 
\begin{rem}  \label{rem:borel1} The set $\mathcal{M}_+(C)$ is a convex cone. In symbols,
 if $\mu, \nu\in \mathcal{M}_+(C)$ and $a,b \in \RR_{\geq 0}$, then
$a \cdot \mu + b \cdot \nu \in \mathcal{M}_+(C)$. 
\end{rem}

Given two measures $\mu, \nu\in \mathcal{M}_+(C)$, one defines their \emph{convolution} $\,\mu*\nu\,$ as follows:
\begin{align}\label{eq:conv1}
  (\mu*\nu)(E) \,\,\,= \,\, \int_{\R^n}\int_{\R^n} \chi_{E}(y+z)\,d\mu(y)d\nu(z),
\end{align}
where $E\subset (\R^n)^*$ is any Borel subset and $\chi_E$ denotes its characteristic function. 

\begin{lem} \label{lem:borel2} If $\mu, \nu\in \mathcal{M}_+(C)$ then $\mu*\nu\in \mathcal{M}_+(C)$. 
The convolution is commutative and associative, that is, $\,\mu*\nu=\nu*\mu\,$ and 
$\,(\mu*\nu)*\xi = \mu*(\nu*\xi)\,$ for $\,\mu, \nu, \xi\in \mathcal{M}_+(C)$.
\end{lem}

\begin{proof}
The first statement holds because the preimage of a compact set under the addition map 
$\, C^*\times C^* \rightarrow C^*,\, (y,z) \mapsto y+z\,$
is compact. This ensures that  $\mu*\nu$ is a locally compact Borel measure supported in $C^*$. 
The second statement, namely
commutativity and associativity of the convolution,
 follows from Tonelli's Theorem.
\end{proof}

We can also check that the convolution product is distributive with respect to addition of measures.
In light of Remark \ref{rem:borel1} and Lemma \ref{lem:borel2}, this means that $\bigl(\mathcal{M}_+(C),+,*\bigr)$ is a semiring.
We turn it into a ring by the following standard construction.
Let $\mathcal{M}(C)=\mathcal{M}_+(C)-\mathcal{M}_+(C)$ denote the set of all $\R$-linear combinations of measures in $\mathcal{M}_+(C)$.
Then we extend the convolution \eqref{eq:conv1} by bilinearity to $\mathcal{M}(C)$. 

We conclude that $\bigl(\mathcal{M}(C) ,+, * \bigr)$ is a commutative  $\RR$-algebra.
If we are given a finite collection of $m$ measures in $\mathcal{M}_+(C)$, then
the subalgebra they generate is the quotient $\RR[z_1,\ldots,z_m]/I$ of a polynomial ring
modulo an ideal $I$. In this representation, the tools of computer
algebra, such as {\em Gr\"obner bases}, can be applied to the study of measures.
 
\begin{exmp} \label{ex:alpha1r}
Let $n=1$ and $C = \RR_{>0}$. Fix rational numbers $\alpha_1,\ldots,\alpha_m > 0$
and let~$\mu_i$ be the measure in $\mathcal{M}_+(\RR_{>0})$ with density $y^{\alpha_i}/\Gamma(\alpha_i)$.
Additive relations among the $\alpha_i$ translate into multiplicative relations among the $\mu_i$.
This follows from Example~\ref{exmp:convthm} below.
The algebra generated by these measures is isomorphic to that
generated by the monomials $x^{-\alpha_1},\ldots,x^{-\alpha_m}$. For instance, if
$m = 3$ and $\alpha_i = i+2$ for $i=1,2,3$ then
$$ \bigl(\RR[\mu_1,\mu_2,\mu_3],+,*\bigr) \,\,\, \simeq \,\,\, \RR[x^{-3},x^{-4},x^{-5}] \,\, \simeq \,\,
\RR[z_1,z_2,z_3]/ \langle z_2^2-z_1z_3\,,\,z_1^2 z_2-z_3^2\, ,\, z_1^3-z_2 z_3 \rangle . $$
\end{exmp}

The convolution product has a nice interpretation in probability theory.

\begin{rem}
Let $X_1$ and $X_2$ be independent random variables with values in the cone $C$, and let
$\mu_1 $ and $\mu_2$ be their probability measures. Then the  probability measure $\mu_{X_1+X_2}$ of their sum $X_1+X_2$ is the convolution of the two measures, that is,
$$\mu_{X_1+X_2} \,\,=\,\, \mu_{X_1}*\mu_{X_2}. $$
The convolution of probability measures corresponds to adding random variables.
\end{rem}

It is instructive to verify this statement when the random variables are discrete.

\begin{exmp}[$n=1$]
Let $X_1$ and $X_2$ be independent {\em Poisson random variables} with parameters $\lambda_1, \lambda_2>0$ respectively.
Thus $\mu_{X_1}$ and $\mu_{X_2}$ are atomic measures supported on the set of nonnegative integers,
with $\,\mu_{X_i}(\{k\}) \,=\, \lambda_i^{k}e^{-\lambda_i}/k! \,$ for $k=0,1,\dots$. Then
  \begin{align*}
    \mu_{X_1+X_2}(\{k\}) \,\, &= \,\, (\mu_{X_1}*\mu_{X_2})(\{k\}) \,\, = 
    \sum_{k_1+k_2=k} \!\! \mu_{X_1}(\{k_1\}) \cdot \mu_{X_2}(\{k_2\}) \\
    &= \,\, \sum_{k_1+k_2=k} \frac{\lambda_1^{k_1}e^{-\lambda_1}}{k_1!}\frac{\lambda_2^{k_2}e^{-\lambda_2}}{k_2!}
    \,\, = \,\, \, \frac{(\lambda_1+\lambda_2)^ke^{-(\lambda_1+\lambda_2)}}{k!}.
  \end{align*}
  Hence $X_1+X_2$ is a Poisson random variable with parameter $\lambda_1+\lambda_2$.
  This fact is well-known in probability. In this example, $C=\R_{>0}$ and 
  $\,\mu_{X_1}, \,\mu_{X_2},\, \mu_{X_1}*\mu_{X_2}\in \mathcal{M}_+(C)$.
\end{exmp}

Let $\mu\in \mathcal{M}_+(C)$. If the integral
\begin{align}\label{eq:Lap}
\mathcal{L}\{\mu\}(x) \,\, := \,\,  \int_{C^*} e^{-\langle y,x\rangle} d\mu(y) 
\end{align}
converges for all $x \in C$, we say that $\mu$ has a \emph{Laplace transform} $\mathcal{L}\{\mu\}:C\rightarrow \R_{>0}$.

\begin{rem}\label{rem:Lap}
By the Bernstein-Hausdorff-Widder-Choquet Theorem \ref{thm:BHWC}, the Laplace transform
 is a completely monotone function on $C$, and if $\mathcal{L}\{\mu\}=\mathcal{L}\{\nu\}$ then $\mu=\nu$.
\end{rem}

The Laplace transform takes convolutions of measures to products of functions.

\begin{prop}\label{prop:conv}
If $\mu, \nu\in\mathcal{M}_+(C)$ have Laplace transforms, then so does $\mu*\nu$, and
\begin{align}\label{eq:convthm}
  \mathcal{L}\{\mu*\nu\}(x)  \,\,=\,\,
   \mathcal{L}\{\mu\}(x) \cdot\mathcal{L}\{\nu\}(x) \quad \hbox{for all} \,\,\,x \in C.
\end{align}
In particular, the product of two completely monotone functions on $C$ is completely monotone, and
its Riesz measure is the convolution of the individual Riesz measures.
\end{prop}

\begin{proof}
The definition in \eqref{eq:conv1} and Tonelli's Theorem imply that, for any $x\in C$,
  \begin{align}
    \mathcal{L}\{\mu*\nu\}(x) &\,\,=\,\, \int_{C^*} e^{-\langle w,x\rangle}\, d(\mu*\nu)(w) \,\,=\,\, \int_{C^*}\left(\int_{C^*} e^{-\langle y,x\rangle}d\mu(y)\right) e^{-\langle z,x\rangle}d\nu(z)\\
    &\,\,=\,\, \int_{C^*} e^{-\langle y,x\rangle}\,d\mu(y)\int_{C^*} e^{-\langle z,x\rangle}\,d\nu(z) 
    \,\,=\,\, \mathcal{L}\{\mu\}(x)\cdot\mathcal{L}\{\nu\}(x).
  \end{align}
  The assertion in the second sentence follows from Theorem \ref{thm:BHWC} and Remark \ref{rem:Lap}.
\end{proof}

\begin{exmp}\label{exmp:convthm} Fix $n=1$ and let
$\mu_i$ be the measure on $\R_{>0}$ with density $y^{\alpha_i-1}/\Gamma(\alpha_i)$.
By equation \eqref{eq:product1}, its Laplace transform is the monomial
$\,  \mathcal{L}\{\mu_i\}(x)  = x^{-\alpha_i}$. Thus the assignment $\mu_i \mapsto \mathcal{L}\{\mu_i\}$
gives the isomorphism of $\RR$-algebras promised in Example~\ref{ex:alpha1r}.
\end{exmp}

Here we are tacitly using the following natural extension of the Laplace transform from the semiring $\mathcal{M}_+(C)$
to the full $\RR$-algebra $ \mathcal{M}(C)$. If $\mu=\mu_+-\mu_-\in \mathcal{M}(C)$,
where the measures $\mu_+,\mu_-\in \mathcal{M}_+(C)$ have Laplace transforms, then
$\mathcal{L}\{\mu\} := \mathcal{L}\{\mu_+\}  -  \mathcal{L}\{\mu_-\} $.

Now, let $\mu_1,\ldots,\mu_m$ be measures in $\mathcal{M}_+(C)$ that have  Laplace transforms.
We write $\RR[\mu_1,\ldots,\mu_m]$ for the $\RR$-algebra they generate with respect to 
the convolution product~$*$. This is a subalgebra of the commutative algebra
$(\mathcal{M}(C),+,*)$. By Proposition \ref{prop:conv}, the Laplace transform is an algebra homomorphism from
$\RR[\mu_1,\ldots,\mu_m]$ 
  into the algebra of $\mathcal{C}^{\infty}$-functions on $C$. 
  Moreover, by Remark \ref{rem:Lap}, this homomorphism has trivial kernel.

This construction allows us to transfer polynomial relations among completely monotone functions
to polynomial relations among their Riesz kernels, and vice versa. In the remainder of this section,
we demonstrate this for the scenario in Section~\ref{sec3}.

Fix linear forms $\ell_1,\ldots, \ell_m\in (\R^n)^*$ and let $C=\{\ell_1>0,\dots,\ell_m>0\}\subset \R^n$ be the polyhedral cone
they define. For all $i=1,\ldots,m$, and all $x$ with $\ell_i(x) >0$, we have
\begin{equation}
\label{eq:rayi}
  \ell_i^{-1}(x) \,\,= \,\, \int_{\R_{\geq 0}\ell_i} \!\! e^{-\langle y,x\rangle} d\mu_i(y) 
  \,\, = \,\, \int_{\R_{\geq 0}} \!\! e^{-t\ell_i(x)}\,dt.
\end{equation}
Let $\mu_i$ be the Lebesgue measure $dt$ on the ray $\R_{\geq 0}\ell_i = \{t\ell_i:t>0\}$, viewed as a measure on $C^*$.
By \eqref{eq:rayi}, the Laplace transform of this measure is the reciprocal linear form:
$$ \mathcal{L}\{\mu_i\} \,\,\, = \,\,\,  \ell_i^{-1} . $$
The subalgebra of $\mathcal{M}(C)$ generated by $\mu_1,\dots, \mu_m$ is isomorphic, via Laplace transform, to the subalgebra of the algebra of rational functions on $\R^n$ generated by $\ell_1^{-1},\dots, \ell_m^{-1}$.
This algebra was introduced in \cite{OrlikTerao}. It
is known as the {\em Orlik-Terao algebra} of $\ell_1,\dots,\ell_m$.

\begin{cor}
The convolution algebra  $\RR[\mu_1,\ldots,\mu_m] $ is isomorphic to the
Orlik-Terao algebra, and therefore to
$ \RR[z_1,\ldots,z_m]/I$, where $I$ is the Proudfoot-Speyer ideal in \cite{PrSp}.
Its monomials $\mu_{i_1} {*} \mu_{i_2} {*} \cdots {*} \mu_{i_s}$, where
$\ell_{i_1},\ell_{i_2},\ldots,\ell_{i_s}$ runs over multisubsets of linear forms that span $(\RR^n)^*$,
are the piecewise polynomial volume functions in Theorem \ref{thm:prodlin}.
\end{cor}

Indeed, Proudfoot and Speyer \cite{PrSp} gave an excellent
presentation of the Orlik-Terao algebra by showing
 that the circuit polynomials 
form a universal Gr\"obner basis of $I$.

\begin{exmp}[$n{=}3,m{=}5$] Let $\ell_1,\ldots,\ell_5$ be the linear forms in Example~\ref{ex:3by5}.
Then 
\begin{equation}
\label{eq:conv35}
\RR[\mu_1,\ldots,\mu_5] \,\,\, = \,\,\, \RR[z_1,\ldots,z_5]/I, 
\end{equation}
where the Proudfoot-Speyer ideal $I$ is generated by its universal Gr\"obner basis
$$ \begin{matrix} \bigl\{ 
z_1 z_2 z_3-2 z_1 z_2 z_4+3 z_1 z_3 z_4-2 z_2 z_3 z_4 \, , \,\,
z_1 z_2 z_3-z_1 z_2 z_5+2 z_1 z_3 z_5-2 z_2 z_3 z_5, \\ \quad
2 z_1 z_2 z_4-z_1 z_2 z_5+z_1 z_4 z_5-2 z_2 z_4 z_5 \, , \,\,
3 z_1 z_3 z_4-2 z_1 z_3 z_5+z_1 z_4 z_5-2 z_3 z_4 z_5, \\
z_2 z_3 z_4-z_2 z_3 z_5+z_2 z_4 z_5-z_3 z_4 z_5
\bigr\}. \end{matrix} $$
Note that these five cubics are the circuits in $I$.
The monomial $z_1 z_2 z_3 z_4 z_5$ in the convolution algebra \eqref{eq:conv35}
represents the piecewise quadratic function $q(y)$ in
Example~\ref{ex:35}.
\end{exmp}

\section{Elementary Symmetric Polynomials} \label{sec6}

In this section we study complete monotonicity of inverse powers of the elementary symmetric polynomials
$E_{m,n}$. In Theorem \ref{thm:elem} we prove Conjecture \ref{conj:MSUZ} for this
 special class of hyperbolic polynomials. In Theorem \ref{thm:nec} we prove 
 the only if direction of Conjecture \ref{conj:SS}. These results 
resolve questions raised by Scott and Sokal in \cite{scott2014}.

Our first goal is to show that sufficiently negative powers of $E_{m,n}$ are completely~monontone.
We begin with a lemma by Scott and Sokal which is derived from Theorem \ref{thm:BHWC}.

\begin{lem}[Lemma $3.3$ in \cite{scott2014}]\label{lem:rowno} Fix
$\alpha>0$, an open convex cone $C \subset \R^n$,
and~$\mathcal{C}^\infty$ functions $A,B:C\rightarrow \R_{>0}$. The function 
$(x,y) \mapsto (A(x)+B(x) y)^{-\alpha}$ is completely~monotone on $C\times \R_{>0}$ if and only if $B^{-\alpha} e^{-tA/B}$ is completely monotone on $C$ for all $t\geq 0$.
\end{lem}

\begin{rem}
If $A+By$ is a homogeneous polynomial, then both conditions above are equivalent to complete monotonicity of $B^{-\alpha} e^{-tA/B}:C\rightarrow \R_{>0}$ for $t=0$ and some $t>0$.
\end{rem}

We also need the following generalization of Lemma $3.9$ in \cite{scott2014}. 

\begin{lem}\label{lem:podstaw}
Fix two  cones $C$ and $C'$ and $\alpha>0$.
Let $A, B$ be $\mathcal{C}^\infty$ functions on $C'$ such that $g=(A+By)^{-\alpha}$ is completely monotone on $C'\times \R_{>0}$ with Riesz kernel $q^\prime$.
 Let $f(x,r)$ be a completely monotone function on $C\times \R_{>0}$ with Riesz kernel $q$.
Then the function $B^{-\alpha}f(  x,A/B)$ is completely monotone on  $C\times C'$,
with Riesz kernel
\begin{equation}
\label{eq:desiredriesz}
 (z,w) \,\,\mapsto \,\, \int_{\R_{\geq 0}}\frac{\Gamma(\alpha)}{s^{\alpha-1}}q ( z,s)  q'( w,s)  ds .
 \end{equation}
\end{lem}

\begin{proof}
By Theorem \ref{thm:BHWC}, our functions admit the following integral representations:
$$
f(  x,r) =\int_{C^* \times \R_{\geq 0}}  \!\!\!\! e^{- \langle  z,  x \rangle -s\cdot r}\, q(  z,s)d  z ds,\quad
(A+By)^{-\alpha} =\int_{C'^* \times \R_{\geq 0}} \!\!\!\!\!\! e^{- \langle  w,  u \rangle -s\cdot y}\, q'(  w,s)d  w ds.
$$
From the first equation we get
\begin{equation}
\label{eq:weget}
  B^{-\alpha}f(  x,A/B) \quad = \quad
  \int_{C^*\times \R_{\geq 0}} \!\!\!\! e^{- \langle z,  x \rangle }(B^{-\alpha}e^{-s\cdot A/B})\, q (  z,s)d  z ds.
\end{equation}
For fixed $A,B$ we have $(A+By)^{-\alpha}= 
\Gamma(\alpha)^{-1} \int_{\R_{\geq 0}} e^{-s\cdot y} s^{\alpha-1}B^{-\alpha}e^{-s\cdot A/B}ds$. 
This follows from \eqref{eq:product1} by setting $x=A+By$ and changing the variable of integration to
$s/B$. By comparing the two integral representations of $(A+By)^{-\alpha}$,
and by using the injectivity of the Laplace transform, we find
$$ B^{-\alpha}e^{-s\cdot A/B}\,\, = \,\, \int_{C'^*}e^{- \langle  w,  u \rangle}\frac{\Gamma(\alpha)}{s^{\alpha-1}} q'(  w,s)d  w .$$
Substituting this expression into \eqref{eq:weget},
we obtain
$$B^{-\alpha}f(  x,A/B)\,\,=\,\, \int_{C^*\times \R_{\geq 0}}e^{-\langle z,  x \rangle}\left(\int_{C'^*}e^{- \langle w,  u \rangle}\frac{\Gamma(\alpha)}{s^{\alpha-1}} q'(  w,s)d  w\right)\, q (  z,s)d  z ds.$$
All functions we consider are nonnegative, so we can apply  Tonelli's Theorem and get
$$B^{-\alpha}f(  x,A/B)\,\,=\,\,\int_{C^*\times C'^*}e^{- \langle  z,  x \rangle - \langle  w,  u \rangle }
\left(\int_{\R_{\geq 0}}\frac{\Gamma(\alpha)}{s^{\alpha-1}} q'(  w,s) q (  z,s)ds\right)\,d  w d  z.$$
The parenthesized expression is the desired Riesz kernel in \eqref{eq:desiredriesz}.
\end{proof}

We are now ready to prove Conjecture \ref{conj:MSUZ} for elementary symmetric polynomials. Our proof is constructive, 
i.e.,~it yields an explicit formula for the associated Riesz kernel \eqref{eq:Riesz}.   However, the construction is quite complicated, as Example \ref{ex:explicit} shows.

\begin{thm}\label{thm:elem}
For any elementary symmetric polynomial $E_{m,n}$, where $1\leq m\leq n$, there exists
a real number $\alpha'>0$ such that $E_{m,n}^{-\alpha}$ is completely monotone for all $\alpha\geq \alpha'$.
\end{thm}

\begin{proof} 
If $m=1$ or $m=n$, then $E_{m,n}^{-\alpha}$ is completely monotone on $\R_{>0}^n$ for any $\alpha\geq 0$ (see, e.g., Proposition \ref{prop:LapTrProdxi}). Also, by \cite[Corollary 1.10]{scott2014}, $E_{2,n}^{-\alpha}$ is completely monotone for $\alpha\geq (n-2)/2$. For $2<m<n$ we proceed by induction on $m$. We have 
\begin{align}\label{eq:ind}
E_{m,n} \,\,= \,\, E_{m,n-1}\,+ \, E_{m-1,n-1} \cdot y
\end{align}
where $y=x_n$ and the other variables in (\ref{eq:ind}) are $x_1,\ldots,x_{n-1}$. We apply
Lemma \ref{lem:rowno}. 

We must prove that there exists $\alpha_{m,n}>0$ such that, for all
$\,\alpha \geq\alpha_{m,n}\,$ and all $\,t \geq 0$, 
\begin{align}\label{eq:dobro}
E_{m-1,n-1}^{-\alpha} e^{-t E_{m,n-1}/E_{m-1,n-1}}\text{ is completely monotone on } \R_{>0}^{n-1} .
\end{align}
One can derive the following factorization, which holds for any fixed $t \geq 0$:
\begin{align}\label{eq:wnuk}
  E^{-(n-1)\alpha}_{m-1,n-1}e^{-tmE_{m,n-1}/E_{m-1,n-1}} \,\,=\,\, \prod_{i=1}^{n-1}E^{-\alpha}_{m-1,n-1}e^{-tQ_i/E_{m-1,n-1}}, 
\end{align}
where $Q_i=x_iE_{m-1,n-2}(x_1,\dots,\hat{x_i},\dots,x_{n-1})$. The hat means that $x_i$ is omitted.
We  claim that, for each $i$, the function $E_{m-1,n-1}^{-\alpha} e^{-tQ_i/E_{m-1,n-1}}$
 is completely monotone~on $\R_{>0}^{n-1}$ provided $\alpha\geq \max(1/2,\alpha_{m-1,n-1})$. 
 Then, by Proposition \ref{prop:conv}, 
 the product \eqref{eq:wnuk} is completely monotone on $\R_{>0}^{n-1}$
 for $t \geq 0$, and we take $\alpha_{m,n} =(n-1)\max(1/2,\alpha_{m-1,n-1})$.

Now, by symmetry, it suffices to show that $E_{m-1,n-1}^{-\alpha} e^{-tQ_{n-1}/E_{m-1,n-1}}$ is completely monotone.
 This is equivalent, by Lemma \ref{lem:rowno}, to complete monotonicity of $P^{-\alpha}$, where $P=Q_{n-1}+x_nE_{m-1,n-1}$. By \eqref{eq:ind}, $E_{m-1,n-1} = E_{m-1,n-2}+ x_{n-1}E_{m-2,n-2}$. This implies
\begin{equation}
\begin{aligned}
  P &\,\,=\,\, Q_{n-1}+ x_nE_{m-1,n-1} \,\,=\,\, x_{n-1} x_nE_{m-2,n-2}+( x_{n-1}+ x_n)E_{m-1,n-2} \\
  &\,\,=\,\, E_{m-2,n-2}E_{2,3}( x_{n-1}, x_n,E_{m-1,n-2}/E_{m-2,n-2}).
\end{aligned}
\end{equation}
Fix any $\alpha\geq \max(1/2, \alpha_{m-1,n-1})$. 
We apply Lemma \ref{lem:podstaw} to the functions  $f=E_{2,3}^{-\alpha}$ and
$g=E_{m-1,n-1}^{-\alpha} = (E_{m-1,n-2}+ x_{n-1}E_{m-2,n-2})^{-\alpha}$.
By \cite[Corollary 1.10]{scott2014} and
 the induction hypothesis respectively, these are completely monotone.  This implies the claim that
  $\,E_{m-2,n-2}^{-\alpha} E_{2,3}( x_{n-1}, x_n,E_{m-1,n-2}/E_{m-2,n-2})^{-\alpha} \,=\, P^{-\alpha}\,$ 
  is completely monotone.
\end{proof}

The first new case of complete monotonicity concerns large negative powers of $E_{3,5}$.

\begin{exmp}\label{ex:explicit}
We here illustrate our proof of Theorem \ref{thm:elem} by deriving the Riesz kernel for $E_{3,5}^{-\beta}$
from its steps.
By \cite[Corollary 1.10]{scott2014}, the functions $E_{2,4}^{-\alpha}$ and $E_{2,3}^{-\alpha}$ are completely monotone for any $\alpha>1$.
 By \cite[Corollary 5.8]{scott2014}, their Riesz kernels are 
 \begin{equation}
 \begin{matrix}
 &
q_1(\alpha)(y_1,y_2,y_3,y_4) & = & \frac{3^{\frac{3}{2}-\alpha}}{2\pi\Gamma(\alpha)\Gamma(\alpha-1)}\left(E_{2,4}(y_1,y_2,y_3,y_4)-(y_1^2+y_2^2+y_3^2+y_4^2)\right)^{\alpha-2}
\\ {\rm and} & 
q_2(\alpha)(y_1,y_2,y_3) & = & \frac{2^{1-\alpha}}{(2\pi)^{\frac{1}{2}}\Gamma(\alpha)\Gamma(\alpha-\frac{1}{2})}\left(E_{2,3}(y_1,y_2,y_3)-\frac{1}{2}(y_1^2+y_2^2+y_3^2)\right)^{\alpha-\frac{3}{2}}.
\end{matrix}
\end{equation}
We apply Lemma \ref{lem:podstaw} to 
$g=E_{2,4}^{-\alpha}=(x_1x_2+x_1x_3+x_2x_3+ (x_1+x_2+x_3)x_4)^{-\alpha}$ 
with $y=x_4$, and  to $f=E_{2,3}^{-\alpha}=(x_4x_5+(x_4+x_5)x_3)^{-\alpha}$ with $y=x_3$.
We conclude that
\begin{equation}\label{eq:super}
\begin{aligned}
&(x_1+x_2+x_3)^{-\alpha}(x_4x_5+(x_4+x_5)(x_1x_2+x_1x_3+x_2x_3)/(x_1+x_2+x_3))^{-\alpha}\\
=\,\,&(x_1x_2x_4+x_1x_3x_4+x_2x_3x_4+x_5(x_1x_4{+}x_2x_4{+}x_3x_4{+}x_1x_2{+}x_1x_3{+}x_2x_3))^{-\alpha}\\
=\,\, &(Q_4+x_5E_{2,4})^{-\alpha}
\end{aligned}
\end{equation}
is completely monotone for $\alpha>1$. It has the Riesz kernel
\begin{align}
  q(\alpha)(y_1,\dots,y_5) \,\,= \,\,
  \int_{\R_{\geq 0}}\frac{\Gamma(\alpha)}{s^{\alpha-1}}q_1(\alpha)(y_1,y_2,y_3,s)q_2(\alpha)(y_4,y_5,s)ds.
\end{align} 
By Lemma \ref{lem:rowno} applied to \eqref{eq:super}, we derive, for any $t\geq 0$, the complete monotonicity of
$$ E_{2,4}^{-\alpha} \cdot e^{-tQ_4/E_{2,4}}\,\,= \,\, E_{2,4}^{-\alpha} \cdot e^{-t(x_1x_2x_4+x_1x_3x_4+x_2x_3x_4)/E_{2,4}}. $$
Proceeding as in the proof of Lemma \ref{lem:podstaw}, we express its Riesz kernel $R_4(\alpha,t)$ as follows
\begin{align}\label{eq:R_4}
R_{4}(\alpha,t)(y_1,y_2,y_3,y_4)
\,\,= \,\, \frac{\Gamma(\alpha)}{t^{\alpha-1}}q(\alpha)(y_1,y_2,y_3,y_4,t).
\end{align}

In the same way one obtains complete monotonicity, for any $t\geq 0$, of the functions
\begin{align}
 & E_{2,4}^{-\alpha}e^{-tQ_1/E_{2,4}} \,\,=\,\, E_{2,4}^{-\alpha}e^{-t(x_1x_2x_3+x_1x_2x_4+x_1x_3x_4)/E_{2,4}},\\
  & E_{2,4}^{-\alpha}e^{-tQ_2/E_{2,4}} \,\, = \,\, E_{2,4}^{-\alpha}e^{-t(x_1x_2x_3+x_1x_2x_4+x_2x_3x_4)/E_{2,4}},\\
  & E_{2,4}^{-\alpha}e^{-tQ_3/E_{2,4}} \,\, =\,\, E_{2,4}^{-\alpha}e^{-t(x_1x_2x_3+x_1x_3x_4+x_2x_3x_4)/E_{2,4}}.
\end{align}
To obtain their Riesz kernels $R_1(\alpha,t), R_2(\alpha,t),R_3(\alpha,t)$,
we exchange $y_4$ with $y_1,y_2, y_3$ in \eqref{eq:R_4}.
Multiplying the four functions together, we obtain complete monotonicity of
\begin{align}\label{eq:trick}
  E_{2,4}^{-4\alpha}e^{-3tE_{3,4}/E_{2,4}}\,\,= \,\,\prod_{i=1}^4 E_{2,4}^{-\alpha}e^{-tQ_i/E_{2,4}}.
\end{align}
The Riesz kernel of (\ref{eq:trick}) is written as $R(4\alpha,3t)$.
By Proposition \ref{prop:conv}, this is the convolution
of $R_1(\alpha,t)$, $R_2(\alpha,t)$, $R_3(\alpha,t)$ and $R_4(\alpha,t)$.
By applying Lemma \ref{lem:rowno} to \eqref{eq:trick} we derive complete monotonicity of $E_{3,5}^{-\beta}$ for $\beta=4\alpha>4$.
 Lemma \ref{lem:podstaw} yields the formula 
\begin{align}
q_{\beta}(y_1,\dots,y_4,y_5) \,\,=\,\,\frac{y_5^{\beta-1}}{\Gamma(\beta)} R(\beta,y_5)(y_1,\dots,y_4)
\end{align}
for the Riesz kernel of $E_{3,5}^{-\beta}$. We note
that $q_{\beta}$ is symmetric in its five arguments.
\end{exmp}

We now come to our second main result in this section, namely 
the only if direction in Conjecture \ref{conj:SS}. This was posed by Scott and Sokal.
 Note that Conjecture \ref{conj:SS}
holds  for $m=n$  since $E_{n,n}(x)=x_1 x_2 \cdots x_n$, with Riesz kernel 
for all negative powers given in Proposition \ref{prop:LapTrProdxi}.
  We now prove that the condition $\alpha\geq (n-m)/2$ from Conjecture \ref{conj:SS} is necessary for $E_{m,n}^{-\alpha}$ to be completely monotone on the positive orthant $\R^n_{>0}$. 
    
\begin{thm}\label{thm:nec}
Let $2{\leq} m{<}n$. If $E_{m,n}^{-\alpha}$ is completely monotone, then $\alpha=0$ or $\alpha\geq \frac{n-m}{2}$.
\end{thm}
\begin{proof}
The proof is by induction on $m$. The base case $m=2$  was already established in 
\cite[Corollary 1.10]{scott2014}. Assume that $E_{m,n}^{-\alpha}$ is completely monotone on $\R_{>0}^n$. Then
\begin{align}
E_{m,n}^{-\alpha}(x_1,x_2,\dots,x_n) \quad = \quad (x_1 E_{m-1,n-1}(x_2,\dots,x_n)+E_{m,n-1}(x_2,\dots,x_n))^{-\alpha}.  
\end{align}
For large $x_1>0$, the sign of any derivative with respect to $x_2,\dots,x_n$ of the functions $E_{m-1,n-1}^{-\alpha}(x_2,\dots,x_n)$, $x_2,\dots,x_n>0$, and $E_{m,n}^{-\alpha}(x_1,\dots,x_n)$, $x_1,\dots,x_n>0$, is the same (see \cite[Lemma 3.1]{scott2014}). It follows that $E_{m-1,n-1}^{-\alpha}$ is completely monotone.
Hence, by induction, we have $\alpha \geq  \frac{n-1-(m-1)}{2} = \frac{n-m}{2}$.
This completes the proof of Theorem~\ref{thm:nec}.
\end{proof}

\section{Hypergeometric Functions}\label{sec7}

In \eqref{eq:factored} we started with 
$\, f \,= \, p_1^{s_1} p_2^{s_2} \cdots p_m^{s_m} $,
where $p_i$ is a polynomial in $x = (x_1,\ldots,x_n)$.
This expression can be viewed as a function in three different ways.
First of all, it is a function in $x$, with domain $C$.
Second, it is a function in $s = (s_1,\ldots,s_m)$, with domain
a subset of $\RR_{<0}^m$. And, finally, we can view $f$ as 
function in the coefficients of the polynomials $p_i$.
It is this third interpretation which occupies us in this final section.

Let us begin with the case $m=1$ and consider $f = p^s$, for some hyperbolic polynomial
\begin{equation}
\label{eq:zpoly}
 p \quad = \quad \sum_{a \in \mathcal{A}} z_a \cdot x_1^{a_1} x_2^{a_2} \cdots x_n^{a_n} . 
 \end{equation}
Here $\mathcal{A}$ is a subset of $\mathbb{N}^n$ whose elements have a fixed coordinate sum $d = {\rm degree}(f)$.
We fix  $s = - \alpha$ such that $p = f^s$ is completely monotone. 
We assume that $p$ has Riesz kernel $q(z;y)$, which we consider as a function of the coefficient vector $z = (z_a: a \in \mathcal{A})$.

Let $\mathcal{D} = \mathbb{C} \langle z_a , \partial_a  : a \in \mathcal{A} \rangle $ denote the 
{\em Weyl algebra}
on the $| \mathcal{A}|$-dimensional affine space $\mathbb{C}^{\mathcal{A}}$ whose coordinates are the coefficients $z_a$
in \eqref{eq:zpoly}.
We briefly recall (e.g.~from \cite{SST}) the definition of the 
{\em $\mathcal{A}$-hypergeometric system} $H_\mathcal{A}(\beta)$
with parameters $\beta \in \mathbb{C}^n$. 

The system $H_\mathcal{A}(\beta)$ is the left ideal in $\mathcal{D}$
generated by two sets of differential operators:
\begin{itemize}
 \item the $n$ {\em Euler operators} $\,\sum_{a \in \mathcal{A}} a_i z_a \partial_a - \beta_i $, where $i=1,2,\ldots,n$; \smallskip
 \item the {\em toric operators}  $ \,\prod_{a \in \mathcal{A}} \partial_{a}^{u_a} -  \prod_{a \in \mathcal{A}} \partial_{a}^{v_a} $,
 where $u_a,v_a$ are nonnegative integers satisfying $\sum_{a \in \mathcal{A}} (u_a-v_a) a = 0$.
 Here it suffices to take a {\em Markov basis}  \cite{Seth} for~$\mathcal{A}$.
\end{itemize}
It is known (cf.~\cite[Chapter 4]{SST}) that $H_\mathcal{A}(\beta)$ is regular holonomic and its
holonomic rank equals
${\rm vol}(\mathcal{A})$ for generic parameters $\beta$. A 
sufficiently differentiable function on an open subset of
$\mathbb{R}^\mathcal{A}$ or $\mathbb{C}^\mathcal{A}$ 
is called {\em $\mathcal{A}$-hypergeometric} if it is annihilated by
the toric operators. It is {\em $\mathcal{A}$-homogeneous of degree $\beta$}
if it annihilated also by the $n$ Euler operators. 

\begin{prop} \label{prop:Ahyper}
The Riesz kernel $q(z;y)$ of $p^{-\alpha}$ is $\mathcal{A}$-hypergeometric
in the coefficients $z$ of the polynomial $p$ as in \eqref{eq:zpoly}.
However, it is generally not $\mathcal{A}$-homogeneous.
\end{prop}

\begin{proof} We use G{\aa}rding's integral
representation of $q(z;y)$ given in Theorem \ref{thm:Riesz}.
The toric operators annihilate $q(z;y)$ since we can differentiate 
with respect to $z$ under the integral sign.
The fact that the Riesz kernel is generally not 
$\mathcal{A}$-homogeneous in $z$ can be seen from
the explicit formula for quadratic forms $p$ given 
 in \cite[Proposition 5.6]{scott2014}.
 \end{proof}

We now start afresh and develop an alternative approach for products
of linear forms, as in Section~\ref{sec3}.
The history of hypergeometric functions dates back to 17th century, and there are 
numerous possible definitions.   We describe the approach of Aomoto and
Gel'fand \cite{Aomoto1994, Gel'fand1986, GZ1986},
albeit in its simplified form via local coordinates on the Grassmannian.

Consider an $n\times m$ matrix $A=(a_{ij})$, with $n\leq m$, 
where the first $n\times n$ submatrix 
is the  identity. 
Fix complex numbers $\alpha_1,\dots,\alpha_{m}$ that sum to $m-n$. The {\em hypergeometric 
function} with parameters $\alpha_i$
 is the function in the $n(m-n)$ unknowns  $a_{ij}$ defined by:
 \begin{equation}
 \label{eq:definedby}
\Phi(\alpha;a_{ij}) \,\,\,:= \,\,\,\int_{\mathbb{S}^{n-1}} \prod_{i=1}^n (x_i)^{\alpha_i-1}_+\prod_{j=n+1}^m (a_{1j}x_1+\dots+a_{nj}x_n)_+^{\alpha_j-1}\,dx.
\end{equation}
We integrate over the unit sphere $\mathbb{S}^{n-1}\subset\R^n$
against the standard measure $dx$.
Here, $(x)_+ = {\rm max}\{0,x\}$ for $x\in \R$. The integral in \eqref{eq:definedby} is convergent if 
${\rm Re}(\alpha_i)>0$. For other values of $\alpha_i$ the hypergeometric function $\Phi$ is defined via analytic continuation, see \cite{Gel'fand1986} for details.
One checks that the following partial
differential operators annihilate~$\Phi$:
\begin{enumerate}
\item Column homogeneity gives the operators
$ \,\sum_{i=1}^n a_{ij}\partial_{ij} -\alpha_j +1 \,\text{ for }n<j\leq m$.
\item Row homogeneity gives the operators
$\,\sum_{j=n+1}^m a_{ij}\partial_{ij} +\alpha_i \,\text{ for }1\leq i\leq n$.
\item We have the toric operators $\,{\partial_{ij}}{\partial_{i'j'}}-{\partial_{ij'}}{\partial_{i'j}}
\,\, \text{ for }1\leq i,i'\leq n<j,j'\leq m$.
\end{enumerate}
This means that the function $\Phi$ is  $\mathcal{A}$-hypergeometric,
in the sense defined above, if we
take $\mathcal{A}$ to be the vertex set of
the product of standard simplices $\Delta_n \times \Delta_{m-n}$.

\begin{exmp}[$n=2,m=4$] Fix $\alpha_1=2-\alpha_2-\alpha_3-\alpha_4$ and consider the matrix
$$ A\,\,=\,\, \begin{pmatrix} 1&0&a&b\\ 0&1&c&d\\
\end{pmatrix}. $$ 
 The hypergeometric function $\Phi$ is obtained by integrating a product of four functions, each 
 of which is zero on a half-plane. Hence the integrand is supported on a cone $C
 \subset \mathbb{R}^2$ defined by two out of four linear functions. Which functions these are, depends on the values of $a,b,c,d$. The integral over the circle $\mathbb{S}^1$ 
 is an integral over a circular arc, specified by $a,b,c,d$. This can be written
 as an integral over a segment in $\mathbb{R}^2$.
 For instance, consider the range of parameters given by   $0<-\frac{c}{a}<-\frac{d}{b}$
  and $0<c,d$. The boundary lines of the cone $C$ are $x_1=0$ and $ax_1+cx_2=0$. 
This is shown in Figure~\ref{fig:1}.
  
\begin{center}
\includegraphics[scale=0.33]{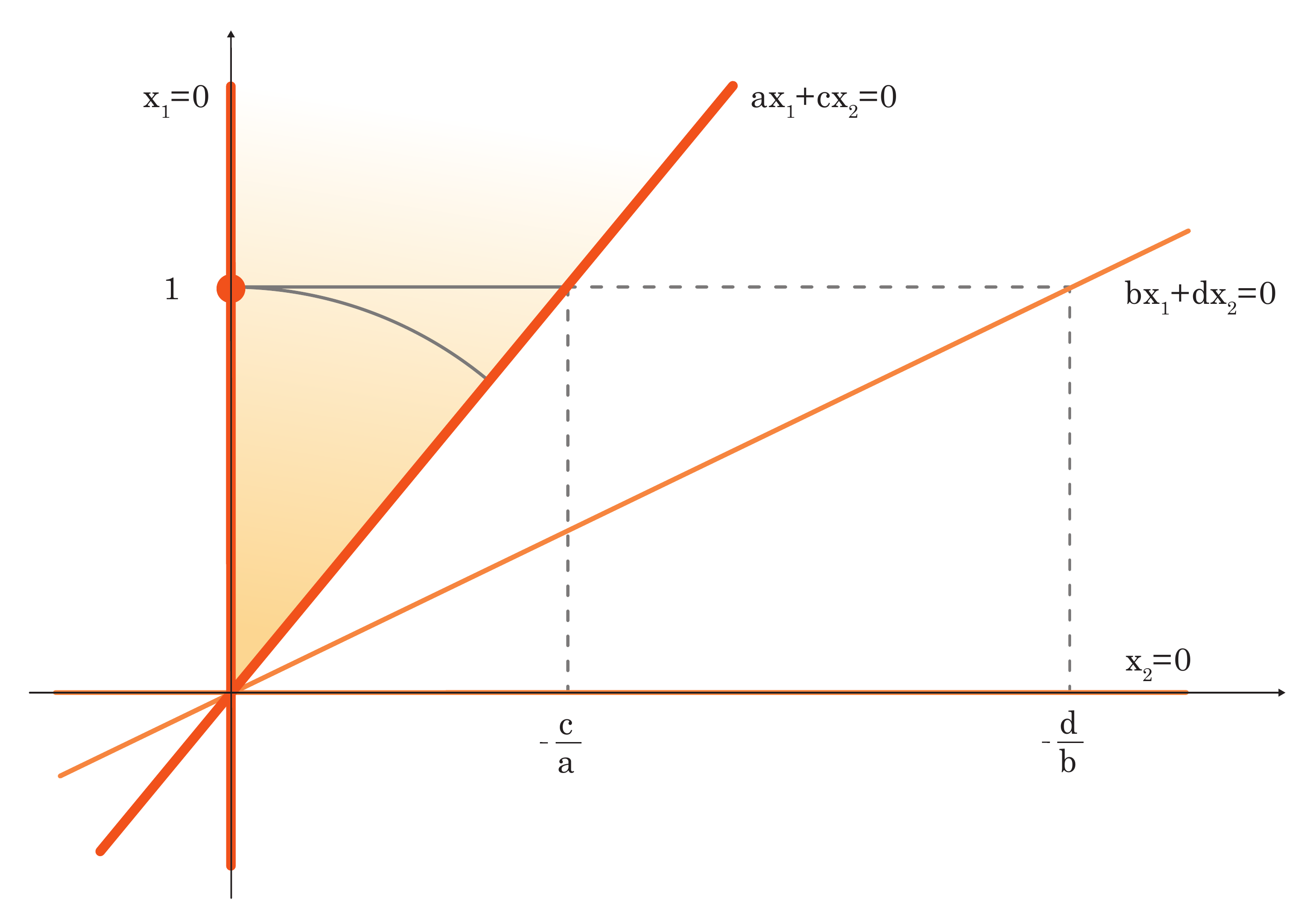}
\vspace{-0.25in}
\captionof{figure}{Four linear forms in two variables. They are positive in the
shaded region. Integration over the unit circle reduces to integration over 
the displayed circular arc.
   We change the integration contour to the horizontal segment.}
\label{fig:1}
\end{center}

Integrating along the segment between $(0,1)$ and $(-\frac{c}{a},1)$, we obtain the  formula
$$\Phi(\alpha_1,\alpha_3,\alpha_4;a,b,c,d) \,\,= \,\, 
\int_0^{-\frac{c}{a}}x_1^{\alpha_1-1}(ax_1+c)^{\alpha_3-1}(bx_1+d)^{\alpha_4-1}dx_1.$$
This integral can be expressed via the classical {\em Gauss  hypergeometric function}
$$_2F_1(a,b,c;z) \,\,:= \,\,\frac{\Gamma(c)}{\Gamma(b)\Gamma(c-b)}\int_0^1 x^{b-1}(1-x)^{c-b-1}(1-zx)^{-a}dx.$$
We assume for simplicity that $c>b>0$.
Performing easy integral transformations, given that
 $\alpha_i$ and $a,b,c,d$ satisfy the assumed inequalities, we obtain:
$$ \begin{matrix} \Phi(\alpha_1,\alpha_3,\alpha_4;a,b,c,d)
\,= \,
(-{a})^{-\alpha_1}c^{\alpha_1+\alpha_3-1}d^{\alpha_4-1}
\frac{ \Gamma(\alpha_1) \Gamma(\alpha_3)}{\Gamma(\alpha_1+\alpha_3)}
 {}_2F_1(1-\alpha_4, \alpha_1,\alpha_1+\alpha_3;\frac{bc}{ad}). \end{matrix} $$
\end{exmp}

Below we present an example involving three 
linear forms in two variables. Here the Riesz kernel is expressed in terms of the classical Gauss 
hypergeometric function $_2F_1$.

\begin{exmp}
Let $p(x)=x_1 x_2(x_1+vx_2)$ with hyperbolicity cone $C=\R_{>0}^2$, where $v>0$.
We consider the function $f = p^{-\alpha}$, where  $\alpha>0$. Then
\begin{align}
  q(y) \,\,=\,\, 
  \begin{cases}
    \frac{(y_2/v)^{2\alpha-1}(vy_1-y_2)^{\alpha-1}}{\Gamma(\alpha)\Gamma(2\alpha)} 
    \,_2F_1\left(1-\alpha,\alpha;2\alpha;\frac{y_2}{y_2-vy_1}\right), & \  \textrm{\normalfont{if}}\ \ 0\leq y_2\leq vy_1,\\
    \quad
 \frac{y_1^{2\alpha-1}(y_2-vy_1)^{\alpha-1}}{\Gamma(\alpha)\Gamma(2\alpha)}
\, \,_2F_1\left(1-\alpha,\alpha,2\alpha;\frac{-vy_1}{y_2-vy_1}\right), & \  \text{\normalfont{if}}\  \ 0\leq y_1\leq y_2/v.
 \end{cases}
\end{align}
\end{exmp}

For the general case, let $\ell_1,\ldots,\ell_m$ be linear forms on $\R^n$
which span a full-dimensional pointed cone $C^*$ in  $(\R^n)^*$.  After a linear
change of coordinates, we can assume that 
 $\ell_{m-n+1},\dots,\ell_m$ is a basis of $(\R^n)^*$,
 and that each other $\ell_i$ has coordinates $y_i = (y_{i1},\ldots,y_{in})$ in that basis.
  Consider the projection $\,\R^m \rightarrow \R^n,\, e_i\mapsto \ell_i$. 
 The kernel of this linear map is spanned by the rows of the following $(m-n)\times m$ matrix:
$$ \begin{small} \begin{pmatrix}
1&0&\cdots&0&-y_1\\
0&1&\cdots&0&-y_2\\
\vdots&&\ddots&&\vdots\\
0&0&\cdots&1&-y_{m-n}\\
\end{pmatrix} . \end{small} $$
We extend the above matrix to an $(m-n+1)\times(m+1)$ matrix by adding a
first column $(1,0,\dots,0)$ and a first row that encodes
the vector $y$ of unknowns:
$$ \begin{small} \begin{pmatrix}
1&0&0&\cdots&0&y\\
0&1&0&\cdots&0&-y_1\\
0&0&1&\cdots&0&-y_2\\
\vdots&\vdots&&\ddots&&\vdots\\
0&0&0&\cdots&1&-y_{m-n}\\
\end{pmatrix}.
\end{small} $$
The $(i+1)$-st column in the above matrix is associated to the linear form $y_i$,
and hence we may associate to it the parameter $\alpha_i$. 
We finally define $\alpha_0$ by the equality $\sum_{i=0}^m\alpha_i= n$. 
The following formula 
gives an alternative perspective on
 Theorem \ref{thm:prodlin}.

\begin{thm} \label{thm:prodlinhyper}
Using the notation above, the Riesz kernel for $\,f = \prod_{i=1}^m \ell_i^{-\alpha_i}$ equals 
$$ q(y) \,\, = \,\,
\frac{\Phi(\alpha; y,-y_1,\dots,-y_{m-n})}{\prod_{i=1}^m\Gamma(\alpha_i)}, $$
where the numerator is the Aomoto-Gel'fand hypergeometric function
defined in \eqref{eq:definedby}.
\end{thm}

\begin{proof}
By the results of Gel'fand and Zelevinsky in \cite{GZ1986}, the function $\Phi$ satisfies
\begin{align}\label{eq:hyp2}
  \Phi(\alpha;y,-y_1,\dots,-y_{m-n})  \,\,= \,\,
  \frac{1}{|L|}
  \int_{L^{-1}(y)} \prod\limits_{i=1}^m x_i^{\alpha_i-1} \,dx.
\end{align}
This derivation is non-trivial.
Here, $L:\R^m_{\geq 0}\rightarrow C^*$ is the linear projection taking the standard basis $e_1,\dots,e_m$ to the linear forms $\ell_1,\dots, \ell_m$. Comparing this expression with the formula for the Riesz kernel given in Theorem \ref{thm:prodlin}, we obtain the claimed result. 
\end{proof}

Theorem \ref{thm:prodlinhyper} serves a blueprint for
other completely monotone functions (\ref{eq:factored}).
We are optimistic that  future formulas for Riesz kernels 
will be inspired by Proposition \ref{prop:Ahyper}.

\bibliographystyle{plain}

\bigskip
\bigskip

\noindent
\footnotesize {\bf Authors' addresses:}

\noindent
Khazhgali Kozhasov,
TU Braunschweig,
\hfill {\tt k.kozhasov@tu-braunschweig.de}

\noindent
Mateusz Micha\l ek,
MPI-MiS Leipzig and Aalto University,
\hfill {\tt mateusz.michalek@mis.mpg.de}

\noindent Bernd Sturmfels,
 \  MPI-MiS Leipzig and
UC  Berkeley, \hfill  {\tt bernd@mis.mpg.de}

\end{document}